\documentclass{ws-m3as}

\usepackage{xcolor}

\newcommand{\nn}{\nonumber}
\newcommand{\eps}{\varepsilon}
\newcommand{\bT}{\mathbf{T}}
\newcommand{\bV}{\mathbf{V}}
\newcommand{\bU}{\mathbf{U}}
\newcommand{\bx}{\mathbf{x}}

\begin{document}

\markboth{W. Bao, Y. Feng and J. Yin}{Improved uniform error bounds for long-time Dirac}

%
\catchline{}{}{}{}{}
%

\title{Improved uniform error bounds on time-splitting methods for the long-time dynamics of the Dirac equation with small potentials}

\author{Weizhu Bao}

\address{Department of Mathematics,
National University of Singapore, Singapore, 119076, Singapore \\
matbaowz@nus.edu.sg}

\author{Yue Feng}

\address{Department of Mathematics,
National University of Singapore, Singapore, 119076, Singapore \\
fengyue@u.nus.edu}

\author{Jia Yin}

\address{Applied Mathematics and Computational Research (AMCR) Division, Lawrence Berkeley National Laboratory, Berkeley, CA 94720, USA \\
	jiayin@lbl.gov}

\maketitle

\begin{history}
\received{(Day Month Year)}
\revised{(Day Month Year)}
\comby{(xxxxxxxxxx)}
\end{history}

\begin{abstract}
	We establish improved uniform error bounds on time-splitting methods for the long-time dynamics of the Dirac equation with small electromagnetic potentials characterized by a dimensionless parameter $\eps\in (0, 1]$ representing the amplitude of the potentials. We begin with a semi-discritization of the Dirac equation in time by a time-splitting method, and then followed by a full-discretization in space by the Fourier pseudospectral method. Employing the unitary flow property of the second-order time-splitting method for the Dirac equation, we prove uniform error bounds at $C(t)\tau^2$ and $C(t)(h^m+\tau^2)$ for the semi-discretization and full-discretization, respectively, for any time $t\in[0,T_\eps]$ with  $T_\eps = T/\eps$ for $T > 0$,  which are uniformly for $\eps \in (0, 1]$, where $\tau$ is the time step, $h$ is the mesh size, $m\geq 2$ depends on the regularity of the solution, and $C(t) = C_0 + C_1\eps t\le C_0+C_1T$ grows at most linearly with respect to $t$ with $C_0\ge0$ and $C_1>0$ two constants independent of  $t$, $h$, $\tau$ and $\eps$. Then by adopting the {\sl{regularity compensation oscillation}} (RCO) technique which controls the high frequency modes by the regularity of the solution and low frequency modes by phase cancellation and energy method, we establish improved uniform error bounds at $O(\eps\tau^2)$ and $O(h^m +\eps\tau^2)$ for the semi-discretization and full-discretization, respectively, up to the long-time $T_\eps$. Numerical results are reported to confirm our error bounds and to demonstrate that they are sharp. Comparisons on the accuracy of different time discretizations for the Dirac equation are also provided.
\end{abstract}



\section{Introduction} \label{sec:intro}
In this paper, we consider the Dirac equation in one or two dimensions (1D or 2D), which can be represented in the two-component form with wave function $\Phi : = \Phi(t, {\mathbf{x}}) = (\phi_1(t, {\mathbf{x}}), \phi_2(t, {\mathbf{x}}))^T \in \mathbb{C}^2$ as\cite{Dirac1,Dirac2,Thaller}
\begin{equation}
\label{eq:Dirac_21}
i\partial_t\Phi =  \left(- i\sum_{j = 1}^{d}
	\sigma_j\partial_j + \sigma_3 \right)\Phi+ \varepsilon \left(V(\mathbf{x})I_2 - \sum_{j = 1}^{d}A_j(\mathbf{x})\sigma_j\right)\Phi, \quad \bx \in \Omega,
\end{equation}
where $\Omega \subset \mathbb{R}^d$ $(d = 1, 2)$ is a bounded domain equipped with periodic boundary condition. Here,  $i = \sqrt{-1}$, $t \ge 0$ is time, $\mathbf{x} = (x_1, \ldots, x_d)^T \in \mathbb{R}^d$, $\partial_j = \frac{\partial}{\partial x_j} (j = 1, \ldots, d)$, $\varepsilon \in (0, 1]$ is a dimensionless parameter, $V(\mathbf{x})$ and $A_j({\mathbf x})$ are the given real-valued time-independent electric and magnetic potentials, respectively. $I_2$ is the $2 \times 2$ identity matrix, and $\sigma_1, \sigma_2, \sigma_3$ are the Pauli matrices defined as
\begin{equation}
\label{Pauli}
\sigma_1 = \begin{pmatrix} 0 &\ \  1\\ 1 &\ \  0 \end{pmatrix}, \quad
\sigma_2 = \begin{pmatrix} 0 & \ \ -i \\ i &\ \   0\end{pmatrix}, \quad
\sigma_3 = \begin{pmatrix} 1 &\ \ 0 \\ 0 &\ \  -1 \end{pmatrix}.
\end{equation}
In order to study the dynamics of the Dirac equation \eqref{eq:Dirac_21}, the initial condition is taken as
\begin{equation}
\label{eq:initial}
\Phi(t=0, \mathbf{x}) = \Phi_0(\mathbf{x}), \quad {\bf x} \in \overline{\Omega}.
\end{equation}

For the Dirac equation \eqref{eq:Dirac_21} with $\varepsilon = 1$, i.e., the classical regime, there are comprehensive analytical and numerical results in the literatures. Along the analytical front, for the existence and multiplicity of bound states and/or standing wave solutions, we refer to Refs.~\refcite{ES,GGT,Gross} and references therein. In the numerical aspect, different numerical methods have been proposed and analyzed, such as the finite difference time domain (FDTD) methods\cite{BCJT,MY}, exponential wave integrator Fourier pseudospectral (EWI-FP) method\cite{BCJT,BCJY}, and time-splitting Fourier pseudospectral (TSFP) method\cite{BCY,BY}, etc. For more details related to the numerical schemes, we refer to Refs.~\refcite{BSG,FLB,Gosse,GSX,MK} and references therein.

However, to our knowledge, much less research has been done for the long-time dynamics when $\eps\in(0, 1]$. With a small $\eps$, we can study the long-time dynamics of the Dirac equation \eqref{eq:Dirac_21} for $t\in [0,T_\eps]$ with  $T_{\varepsilon} := T/\eps$ for $T>0$, while either $\eps$ or $T$ can be fixed (or changed) in the analysis. When $\eps\in(0, 1]$ is fixed, e.g. $\eps=1$, we would like to investigate how the errors of different numerical schemes perform for $t \in [0, T_\eps]$ with increasing $T$. On the other hand, when $T$ is fixed, we are interested in analyzing the dependency of errors for different numerical methods on the parameter $\eps$ for $t \in [0, T_{\varepsilon}]$. In our recent works, we examined the long-time error bounds for finite difference time domain (FDTD) methods and finite difference Fourier pseudospectral (FDFP) methods, where the finite difference was applied to discretize the Dirac equation \eqref{eq:Dirac_21} in time and different discretizations were taken in space. We rigorously proved that the FDTD methods share an error bound at $O(\frac{h^2}{\eps}+\frac{\tau^2}{\eps})$  up to the long-time $T_\eps$ with $h$ the spatial mesh size and $\tau$ the temporal step size, while the FDFP methods exhibit an error bound at $O(h^m+\frac{\tau^2}{\eps})$ up to the long-time\cite{FY}.
The uniform spatial error bound with respect to $\eps$ for the FDFP methods indicate that these methods have better spatial resolution than the FDTD methods to solve the Dirac equation with small potentials in the long-time regime. Nevertheless, both types of methods face severe numerical burdens when $\varepsilon \to 0^+$ due to the $\eps$-dependence of the temporal error. To deal with this problem, we adopted the exponential wave integrator (EWI) method for time discretization, and obtained a uniform error bound
at $O(h^m+\tau^2)$ up to the long-time\cite{FY2}.

To further improve the temporal error for $\eps\in (0, 1]$, we consider the time-splitting methods, which have been widely used to numerically solve dispersive partial differential equations (PDEs)\cite{BaoC,GL,HWL,MQ}. It has been proved that the temporal error bounds grow linearly when the time-splitting methods are applied to the Maxwell's equations\cite{CLL1,CLL2}. In addition, the unitary flow property of the time-splitting methods contributes to the time-dependent uniform error bounds for the Schr\"odinger equation and the improved uniform error bounds for the Schr\"odinger/nonlinear Schr\"odinger equation with the help of the {\bf regularity compensation oscillation (RCO)} technique\cite{BCF1,BCF2}. This improves the analysis in the previous studies and obtains surprising error estimates for the time-splitting methods\cite{BCF1,BCF2}. The aim of this paper is to establish the improved uniform error bounds on time-splitting methods for the long-time dynamics of the Dirac equation with small potentials. First, we prove a uniform error bound where the error constant grows linearly with respect to the time $t$. Based on the error bound, for a given accuracy $\delta_0$ and time step $\tau$, the second order time-splitting (Strang splitting) could be applied to simulate the dynamics of the Dirac equation up to $O(\delta_0/\tau^2)$ when $\eps=1$, i.e., the smaller the time step $\tau$, the longer the dynamics that can be calculated. Then by employing the RCO technique\cite{BCF1,BCF2},  we establish improved uniform error bounds at $O(\eps\tau^2 + \tau_0^m)$ and $O(h^{m}+\eps\tau^2+ \tau_0^m)$ for the semi-discretization and full-discretization, respectively,  for the Dirac equation with $O(\eps)$-potentials up to the long-time at $O(1/\eps)$. In the error bounds, $m \geq 2$ depends on the regularity of the exact solution and $\tau_0 \in (0, 1)$ is a fixed chosen  parameter. When the solution is smooth, i.e. $m\to \infty$, then the error bounds collapse to $O(\eps\tau^2)$ and $O(h^{m}+\eps\tau^2)$ for the semi-discretization and full-discretization, respectively.

The rest of this paper is organized as follows. In Section 2, the second-order time-splitting methods including the semi-discretization and full-discretization for the long-time dynamics of the Dirac equation with small potentials are presented. In Section 3, uniform error bounds for the time-splitting methods are established up to the long-time at $O(1/\eps)$ and the errors are shown to grow linearly with respect to the time $t$. In Section 4, we prove the improved uniform error bounds rigorously by adopting the RCO technique. Extensive numerical results are reported in Section 5. Finally, some conclusions are drawn in Section 6. Throughout this paper, we adopt the notation $A \lesssim B$ to represent that there exists a generic constant $C > 0$, which is independent of the mesh size $h$ and time step $\tau$ as well as the parameter $\varepsilon$ such that $|A| \leq C B$.

\section{The time-splitting methods} \label{sec:TSFP}
For simplicity, in the following sections, we focus on the Dirac equation \eqref{eq:Dirac_21} in 1D, i.e., $d=1$ in \eqref{eq:Dirac_21}, for the numerical methods and corresponding analysis. The methods and results can be easily generalized to \eqref{eq:Dirac_21} in 2D, i.e., $d=2$, and to the four-component Dirac equation given in Refs.~\refcite{BCJT} and \refcite{BCJY}.

The Dirac equation \eqref{eq:Dirac_21} in 1D on the bounded computational domain $\Omega = (a, b)$ with periodic boundary condition collapses to
\begin{align}
\label{eq:Dirac_1D}
&i\partial_t\Phi =  (- i \sigma_1 \partial_x + \sigma_3 )\Phi+ \eps(V(x)I_2 - A_1(x)\sigma_1)\Phi, \qquad x \in \Omega,\quad t > 0,\\
\label{eq:ib}
&\Phi(t, a) = \Phi(t, b),\quad t \geq 0; \qquad  \Phi(0, x) = \Phi_0(x),\quad x \in \overline{\Omega},
\end{align}
where $\Phi := \Phi(t, x)$ and $\Phi_0(a) = \Phi_0(b)$.

\subsection{The semi-discretization}
Define the operators
\begin{equation}\label{eq:operators}
    \bT := -\sigma_1\partial_x - i\sigma_3, \quad \bV := -i(V(x)I_2 - A_1(x)\sigma_1),
\end{equation}
then \eqref{eq:Dirac_1D} can be expressed by
\begin{equation}
    \partial_t\Phi(t, x) = (\bT+\eps\bV)\Phi(t, x), \ x \in \Omega, \ t > 0.
\label{eq:Dirac1d}
\end{equation}

Take a time step size $\tau>0$ and denote the time grids as $t_n=n\tau$ for $n=0,1,\ldots$. As both $\bT$ and $\bV$ are time-independent, the solution to \eqref{eq:Dirac1d} with \eqref{eq:ib} can be propagated through
\begin{equation}\label{eq:prop}
    \Phi(t_{n+1},x) = e^{\tau(\bT + \eps\bV)}\Phi(t_n, x), \quad n = 0, 1, \ldots.
\end{equation}
To approximate the operator $e^{\tau(\bT+\eps\bV)}$, we apply the second-order  time-splitting (Strang splitting)\cite{Strang}, which gives
\begin{equation}
e^{\tau(\bT + \eps\bV)} \approx e^{\frac{\tau}{2}\bT}e^{\eps\tau\bV}e^{\frac{\tau}{2}\bT}.
\end{equation}
Therefore, the semi-discretization of the Dirac equation \eqref{eq:Dirac_1D} via Strang splitting can be expressed as
\begin{equation}\label{eq:S2}
\Phi^{[n+1]}(x) = \mathcal{S}_{\tau}(\Phi^{[n]}(x)) :=  e^{\frac{\tau}{2}\bT}e^{\eps\tau\bV}e^{\frac{\tau}{2}\bT}\Phi^{[n]}(x), \quad n=0, 1, \ldots,
\end{equation}
where we take the initial condition $\Phi^{[0]}(x):=\Phi_0(x)$ for $x\in\overline{\Omega}$. Here, $\Phi^{[n]}(x)$ is the approximation of $\Phi(t_n,x)$.

\subsection{The full-discretization}
By noticing the definition of $\bV$ in \eqref{eq:operators}, it is easy to derive that
\begin{equation}
    e^{\eps\tau\bV}\Phi(t, x) = e^{-i \eps \tau\left(V(x) I_2 - A_1(x) \sigma_1\right)} \Phi(t, x), \quad x \in \overline{\Omega}, \quad t>0.
\end{equation}
On the other hand, to get $e^{\tau\bT}\Phi(t, x)$, we can discretize \eqref{eq:Dirac1d} in space by the Fourier spectral method, and then it is possible to integrate the operator analytically in the phase space. We take $M+1$ uniformly sampled grid points in $\overline{\Omega}$ with $M$ a positive even integer
\begin{equation}
    x_j = a + jh, \quad h = \frac{b-a}{M}, \quad j=0, 1, \ldots, M,
\end{equation}
and we denote the sets $X_M$, $Y_M$, $Z_M$ as
\begin{align*}
  & X_M  = \left\{U = (U_0, U_1, \ldots, U_M)^T~|~ U_j\in\mathbb{C}^2, \, j=0, 1, \ldots, M, \, U_0 = U_M\right\},\\
    & Y_M = Z_M\times Z_M, \quad Z_M = {\rm{span}}\{\phi_l(x)=e^{i\mu_l(x-a)}, \, l\in\mathcal{T}_M\},
\end{align*}
where the index set $\mathcal{T}_M = \{l~|~l=-M/2, -M/2+1, \ldots, M/2-1\}$, and $\mu_l = 2\pi l/(b-a)$ for $l\in\mathcal{T}_M$. The projection operator $P_M: \left(L^2(\Omega)\right)^2\rightarrow Y_M$ is defined as
\begin{equation*}
    \left(P_MU\right)(x) := \sum_{l\in\mathcal{T}_M}\widehat{U}_le^{i\mu_l(x-a)}, \quad U(x)\in \left(L^2(\Omega)\right)^2,
\end{equation*}
where
\begin{equation}
    \widehat{U}_l = \frac{1}{b-a}\int_a^bU(x)e^{-i\mu_l(x-a)}dx, \quad l\in\mathcal{T}_M;
\end{equation}
and by taking $\left(C_{\rm per}(\overline{\Omega})\right)^2 = \left\{U \in \left(C(\overline{\Omega})\right)^2~|~U(a) = U(b) \right\}$, the interpolation operator $I_M : \left(C_{\rm per}(\overline{\Omega})\right)^2\rightarrow Y_M$ or $I_M: X_M\rightarrow Y_M$ is defined as
\begin{equation*}
    \left(I_MU\right)(x) := \sum_{l\in\mathcal{T}_M}\widetilde{U}_le^{i\mu_l(x-a)}, \qquad  U(x)\in \left(C_{\rm per}(\overline{\Omega})\right)^2\quad {\rm{or}}\quad U\in X_M,
\end{equation*}
where
\begin{equation}
    \widetilde{U}_l = \frac{1}{M}\sum_{j=0}^{M-1}U_je^{-2ijl\pi/M}, \quad l\in\mathcal{T}_M.
\end{equation}
Here we take $U_j = U(x_j)$ if $U$ is a function.

Let $\Phi_j^n$ be the numerical approximation of $\Phi(t_n, x_j)$ and denote $\Phi^n = \left(\Phi_0^n, \Phi_1^n, \ldots, \Phi_M^n\right)^T\in X_M$ as the solution vector at $t=t_n$. Take the initial value $\Phi_j^0=\Phi_0(x_j)$ for $j=0, \ldots, M$, then the time-splitting Fourier pseudospectral (TSFP) method for discretizating the Dirac equation \eqref{eq:Dirac_1D} is given as
\begin{equation}
\begin{split}
&\Phi^{(1)}_j = \sum_{l \in \mathcal{T}_M} e^{-i\frac{\tau \Gamma_l}{2}}\widetilde{(\Phi^n)}_l\; e^{i\mu_l(x_j-a)} =  \sum_{l \in \mathcal{T}_M} Q_l\; e^{-i\frac{\tau D_l}{2}}(Q_l)^T \widetilde{(\Phi^n)}_l\; e^{\frac{2ijl\pi}{M}},\\
& \Phi^{(2)}_j  = e^{-i\varepsilon\tau G(x_j)} \Phi^{(1)}_j = P e^{-i\varepsilon \Lambda_j} P^T \Phi^{(1)}_j,\\
&\Phi^{n+1}_j = \sum_{l \in \mathcal{T}_M} e^{-i\frac{\tau \Gamma_l}{2}}\widetilde{(\Phi^{(2)})}_l\; e^{i\mu_l(x_j-a)} =  \sum_{l \in \mathcal{T}_M} Q_l\; e^{-i\frac{\tau D_l}{2}}(Q_l)^T \widetilde{(\Phi^{(2)})}_l\; e^{\frac{2ijl\pi}{M}},
\end{split}
\label{eq:TSFP}
\end{equation}
for $n=0,1,\ldots$, where $\Gamma_l = \mu_l\sigma_1 + \sigma_3 = Q_l D_l(Q_l)^T$ with $\delta_l =\sqrt{1+\mu^2_l}$, $(Q_l)^T$ is the transpose of $Q_l$ and
\begin{equation}\label{eq:Qdef}
\Gamma_l = \begin{pmatrix} 1 &\   \mu_l \\ \mu_l &\   -1 \end{pmatrix},\quad
Q_l = \begin{pmatrix} \frac{1+\delta_l}{\sqrt{2\delta_l(1+\delta_l)}}  & \  -\frac{\mu_l}{\sqrt{2\delta_l(1+\delta_l)}} \\ \frac{\mu_l}{\sqrt{2\delta_l(1+\delta_l)}} &\    \frac{1+\delta_l}{\sqrt{2\delta_l(1+\delta_l)}} \end{pmatrix},\quad
D_l = \begin{pmatrix} \delta_l &\ 0 \\ 0 &\  -\delta_l \end{pmatrix},
\end{equation}
and $G(x_j) = V(x_j) I_2 - A_1(x_j) \sigma_1 = P \Lambda_j P^T$ with  $\Lambda_j = \textrm{diag}(\Lambda_{j, +},\Lambda_{j, -})$ and $\Lambda_{j, \pm} = V(x_j) \pm A_1(x_j)$, $P = I_2$ if $A_1(x_j) = 0$ and otherwise
\begin{equation*}
P = \begin{pmatrix} \frac{1}{\sqrt{2}} \ & \frac{1}{\sqrt{2}} \\ -\frac{1}{\sqrt{2}} \ &\frac{1}{\sqrt{2}}\end{pmatrix}.
\end{equation*}

\begin{remark}	
If the electromagnetic potential $V$ and/or  $A_1$ are time-dependent, the time-ordering technique\cite{Yin} should be applied when we implement \eqref{eq:S2}, as the operator $\bT+\eps\bV$ is no longer commutable for different time coordinates $t_1\neq t_2$.
\end{remark}

\section{Uniform error bounds} \label{sec:unibds}
In this section, we prove the uniform error bounds for the second-order time-splitting method in propagating the Dirac equation with small potentials in the long-time regime up to $T_\eps = T/\eps$ for any given $T>0$. We will start with the results for the semi-discretized scheme, and then extend it to the full-discretization.

\subsection{For semi-discretization}
Suppose there exists a positive integer $m \geq 2$, such that for the potentials, we have
\begin{equation*}
\textrm{(A)} \qquad\qquad  V(x) \in {W^{m, \infty}_{\rm per}}(\Omega),\quad A_1(x) \in {W^{m, \infty}_{\rm per}}(\Omega),
\end{equation*}
where $W^{m, \infty}_{\rm per}(\Omega):= \{u|u\in W^{m, \infty}(\Omega), \,\partial^l_xu(a) = \partial^l_x u(b), \, l=0,\, \ldots, \,m-1\}$.
In addition, we assume that the exact solution $\Phi : = \Phi(t, x)$ of the Dirac equation \eqref{eq:Dirac_1D} up to the long-time $T_\eps = T/\eps$ satisfies
\begin{equation*}
\textrm{(B)} \qquad \qquad \|\Phi\|_{L^{\infty}([0, T_\eps]; (H_{\rm per}^m(\Omega))^2)}	\lesssim 1,
\end{equation*}
where similarly, $H^{m}_{\rm per}(\Omega):= \{u|u\in H^{m}(\Omega), \,\partial^l_xu(a) = \partial^l_xu(b), \, l=0,\, \ldots, \,m-1\}$, with the equivalent $H^m$-norm on $H^m_{\rm per}(\Omega)$ given as
\begin{equation}
\left\|\phi\right\|_{H^m} = \left(\sum\limits_{l \in \mathbb{Z}}(1+\mu_l^{2})^m|\widehat{\phi}_l|^2\right)^{1/2}.
\end{equation}
Taking into account the assumptions (A) and (B), we can easily derive
\begin{equation}
	\|\partial_{t}\Phi\|_{L^{\infty}([0, T_\eps]; (H_{\rm per}^{m-1}(\Omega))^2)} \lesssim 1,
\end{equation}
from the Dirac equation \eqref{eq:Dirac_21}.
By taking the semi-discretized second-order time-splitting given by \eqref{eq:S2} for the Dirac equation \eqref{eq:Dirac1d} with the operators $\bT$ and $\bV$ defined in \eqref{eq:operators}, we have the following error estimate.
\begin{theorem}
Let $\Phi^{[n]}(x)$ be the numerical approximation obtained from the semi-discretized second-order time-splitting \eqref{eq:S2} for the Dirac equation \eqref{eq:Dirac1d}, then under the assumptions (A) and (B), for any $0 < \varepsilon \leq 1$, we have the uniform error estimate
\begin{equation}
\|\Phi(t_n, x) - \Phi^{[n]}(x)\|_{L^2} \leq  C_1\eps t_n \tau^2\leq C_1T\tau^2, \quad 0 \leq n \leq \frac{T/\varepsilon}{\tau},
\label{eq:TSFP_error1}	
\end{equation}
where $C_1$ is a positive constant independent of $h, \tau, n$ and $\varepsilon$.
\label{thm:S2}
\end{theorem}

\begin{proof}
We notice that $\bT$ generates a unitary group in $(H^k_{\rm per}(\Omega))^2$ $(k \geq 0)$. Denote the exact solution flow $\Phi(t_n) \to \Phi(t_{n+1})$  as
\begin{equation}
\Phi(t_{n+1}) = \mathcal{S}_{e, \tau}(\Phi(t_n)),\quad 0 \leq n \leq \frac{T/\varepsilon}{\tau},
\end{equation}
where we take $\Phi(t_n):=\Phi(t_n, x)$ for simplicity.

In order to prove the convergence, we adopt the approach via formal Lie calculus introduced in Ref.~\refcite{Lubich} and split the proof into the following two steps.

\noindent
\textbf{Step 1} (Bounds for local truncation error). We begin with the local truncation error, i.e., to estimate the error generated by one time step computed via \eqref{eq:S2}. By using Taylor expansion for $e^{\eps\tau \bV}$, we have
\begin{equation*}
\mathcal{S}_{\tau}(\Phi_0) =  e^{\tau{\bT}}\Phi_0 + \eps \tau e^{\frac{\tau{\bT}}{2}}{\bV} e^{\frac{\tau{\bT}}{2}}\Phi_0 + \eps^2 \tau^2\int^1_0 (1-\theta) e^{\frac{\tau{\bT}}{2}}e^{\eps\theta\tau{\bV}}{\bV}^2 e^{\frac{\tau{\bT}}{2}}\Phi_0 d\theta.
\end{equation*}
By Duhamel's principle, we can write
\begin{align*}
\mathcal{S}_{e, \tau}(\Phi_0) = & \ e^{\tau{\bT}}\Phi_0 + \eps\int^\tau_0 e^{(\tau - s) \bT}{\bV} e^{s\bT}\Phi_0 ds \\
 & +\eps^2 \int^{\tau}_0\int^s_0 e^{(\tau-s)\bT} {\bV} e^{(s-w) \bT}{\bV} \Phi(w)d w d s.
\end{align*}
Denote
\begin{equation}
Y(s) = e^{(\tau-s) \bT}{\bV} e^{s\bT}\Phi_0,\quad B(s, w) = e^{(\tau-s) \bT}{\bV} e^{(s-w)\bT}{\bV} e^{w\bT} \Phi_0,
\label{eq:YB}
\end{equation}
then the local truncation error can be written as
\begin{align*}
\mathcal{S}_{\tau}(\Phi_0) - \mathcal{S}_{e, \tau}(\Phi_0) =&\ \eps\tau Y\left(\frac{\tau}{2}\right) - \eps\int^{\tau}_0 Y(s) ds  + \frac{\varepsilon^2\tau^2}{2}B\left(\frac{\tau}{2}, \frac{\tau}{2}\right) \\
& - \eps^2\int^{\tau}_0\int^s_0 B(s, w) ds d w + \eps^2 R_1 + \eps^2 R_2,	
\end{align*}
with
\begin{equation*}
\begin{split}
& R_1 = \tau^2\int^1_0(1-\theta) e^{\frac{\tau{\bT}}{2}}e^{\eps\theta\tau{\bV}} {\bV}^2 e^{\frac{\tau{\bT}}{2}}\Phi_0 d \theta - \frac{\tau^2}{2}B\left(\frac{\tau}{2}, \frac{\tau}{2}\right), \\
& R_2 = -\int^{\tau}_0 \int^s_0 e^{(\tau-s) \bT} {\bV} e^{(s-w)\bT} {\bV}\Phi(w) - B(s, w) d w d s.
\end{split}	
\end{equation*}
It is easy to check that
\begin{align*}
\|R_1\|_{L^2} & \lesssim \tau^2 \max_{\theta \in (0, 1)} \{\|\partial_{\theta\theta}((1-\theta)e^{\frac{\tau{\bT}}{2}}e^{\eps\theta\tau{\bV}} {\bV}^2 e^{\frac{\tau{\bT}}{2}}\Phi_0)\|_{L^2}\} \\
 & \lesssim \eps \tau^3 \left(\| {\bV}^3\Phi_0\|_{L^2} + \varepsilon\tau \| {\bV}^4\Phi_0\|_{L^2}\right) \lesssim \varepsilon \tau^3.
\end{align*}
In view of the properties of $\bT$ and $B(s, w)$, the quadrature rule implies
\begin{align*}
\|R_2\|_{L^2} & \lesssim \tau^3 \max_{s, w \in (0, \tau)}\{\|e^{(\tau-s){\bT}}{\bV} e^{(s-w)\bT} {\bV} \partial_w \Phi(w)\|_{L^2}\} \nonumber \\
& \lesssim \tau^3 \|\partial_s \Phi(\cdot)\|_{L^{\infty}([0, \tau];({L^2})^2)} \lesssim \tau^3,
\end{align*}
and
\begin{equation*}
\left\|\frac{\tau^2}{2}B\left(\frac{\tau}{2}, \frac{\tau}{2}\right) - \int^{\tau}_0\int^s_0 B(s, w) d w d s\right\|_{L^2}	\lesssim \tau^3 \max_{0 \leq w \leq s \leq \tau}(\|\partial_s B\|_{L^2} + \|\partial_w B\|_{L^2}) \lesssim \tau^3.
\end{equation*}
Finally, we estimate the last term,  which contains the major part of the local error
\begin{equation}
\mathcal{F}(\Phi_0):= \eps\tau {Y}(\frac{\tau}{2}) - \varepsilon\int^{\tau}_0 Y(s) d s = \varepsilon\tau^3\int^1_0 \text{ker}(\theta) {Y}''(\theta \tau) d\theta,
\label{eq:malF}
\end{equation}
where $\text{ker}(\theta)$ is the Peano kernel for midpoint rule. In addition, we have
\begin{equation*}
Y''(s) = e^{(\tau-s)\bT} [{\bT}, [\bT, {\bV}]]e^{s\bT}\Phi_0.
\end{equation*}
For the double commutator $ [{\bT}, [{\bT}, {\bV}]]$, we have
\begin{equation*}
\left\|[{\bT}, [\bT, {\bV}]]\Psi\right\|_{L^2}	\lesssim \left(\|V(\cdot)\|_{W^{2, \infty}} + \|A_1(\cdot)\|_{W^{2, \infty}}\right) \|\Psi\|_{H^1}.
\end{equation*}
Combining all the results above, we find the one step local error as
\begin{equation}
\mathcal{S}_{\tau}(\Phi_0) - \mathcal{S}_{e, \tau}(\Phi_0) = \mathcal{F}(\Phi_0) + \mathcal{R}^0,	
\label{eq:diff1}
\end{equation}
where $\|\mathcal{R}^0\|_{L^2} \lesssim \eps^2 \tau^3$ and
\begin{equation}
\mathcal{F}(\Phi_0) = \eps\tau^3\int^1_0 \rm{ker}(\theta)e^{\tau(1-\theta){\bT}}[{\bT}, [{\bT},{\bV}]]	e^{\tau\theta {\bT}}\Phi_0 d \theta.
\end{equation}
Define the local truncation error at $t_n$ for $0 \leq n \leq \frac{T/\varepsilon}{\tau}-1$ as
\begin{equation}\label{eq:CalE}
\mathcal{E}^n(x) = \mathcal{S}_{\tau}(\Phi(t_{n}, x)) - \mathcal{S}_{e, \tau}(\Phi(t_{n}, x)),\quad a \leq x \leq b,	
\end{equation}
then from the above computation, we can get
\begin{equation}
\mathcal{E}^n(x) = \mathcal{F}(\Phi(t_{n})) + \mathcal{R}^n, \quad a \leq x \leq b,\quad 0 \leq n \leq \frac{T/\varepsilon}{\tau}-1,	
\label{eq:E_local}
\end{equation}
where for $0 \leq n \leq \frac{T/\varepsilon}{\tau}-1$,
\begin{equation}
\|\mathcal{F}(\Phi(t_{n}))\|_{L^2} \lesssim \varepsilon \tau^3 \|\Phi(t_{n})\|_{H^1}, \quad \|\mathcal{R}^n\|_{L^2} \lesssim \varepsilon^2\tau^3.
\label{eq:E_local2}
\end{equation}

\noindent
\textbf{Step 2} (Bounds for the global error).  We are going to prove the error bound \eqref{eq:TSFP_error1}. Denote ${\bf e}^{[n]}(x)= \Phi^{[n]} - \Phi(t_n) $, then $\|{\bf e}^{[0]}(x)\|_{L^2} =0$ by definition.  For $0 \leq n \leq \frac{T/\varepsilon}{\tau}-1$, we have
\begin{equation}
{\bf e}^{[n+1]} =  \mathcal{S}_{\tau}(\Phi^{[n]}) -  \mathcal{S}_{\tau}(\Phi(t_n)) +  \mathcal{S}_{\tau}(\Phi(t_n)) -\mathcal{S}_{e, \tau}(\Phi(t_n)).
\label{eq:semi_grow}
\end{equation}
By the error bound \eqref{eq:E_local2} for the local truncation error, we obtain for $0 \leq n \leq \frac{T/\varepsilon}{\tau}-1$,
\begin{equation}
\left\|{\bf e}^{[n+1]}\right\|_{L^2} \leq \left\|{\bf e}^{[n]}\right\|_{L^2} + C_1 \eps \tau^3,
\label{eq:eg}
\end{equation}
where $C_1>0$ is a constant and independent of $h, n, \tau$ and $\varepsilon$.
Then it is straightforward to derive
\begin{equation}
    \left\|{\bf e}^{[n+1]}\right\|_{L^2} \leq \left\|{\bf e}^{[0]}\right\|_{L^2} + C_1(n+1)\eps\tau^3 \leq C_1\eps t_{n+1}\tau^2,
\end{equation}
which completes the proof of Theorem \ref{thm:S2}.
\end{proof}

\begin{remark}
According to Theorem \ref{thm:S2}, the uniform error bound on the time-splitting method for the Dirac equation  grows linearly with respect to $t$. In fact, given an accuracy bound $\delta_0 > 0$, the time (for simplicity, assume $\eps=1$ here) for the second-order splitting method to violate the accuracy requirement  $\delta_0$ is $O(\delta_0/\tau^2)$. The results can be extended to other time-splitting methods. For the first-order Lie-Trotter splitting and fourth-order compact splitting or partitioned Runge-Kutta splitting (PRK4)\cite{BM,Trotter},   the times are at $O(\delta_0/\tau)$ and $O(\delta_0/\tau^4)$, respectively. In other words, higher order splitting method performs much better in the long-time simulation not only regarding the higher accuracy but also longer simulation time to produce accurate solutions. In addition, extensions to 2D/3D are straightforward.
\end{remark}

\subsection{For full-discretization}
For the full-discretization given in \eqref{eq:TSFP} by the second-order time-splitting method for the Dirac equation \eqref{eq:Dirac1d}, we could further derive the following uniform error estimate.

\begin{theorem}
Let $\Phi^n$ be the approximation obtained from the TSFP \eqref{eq:TSFP} for the Dirac equation \eqref{eq:Dirac1d}. Under the assumptions (A) and (B), for any $0 < \varepsilon \leq 1$, we have
\begin{equation}
\|\Phi(t_n, x) - I_M\Phi^n\|_{L^2} \leq C(t_n)\left( h^m + \tau^2\right),
\qquad   0 \leq n \leq \frac{T/\varepsilon}{\tau},
\label{eq:TSFP_error}	
\end{equation}
where $C(t)=C_0 + C_1 \eps t\le C_0 + C_1 T$ with
$C_0$ and $C_1$ two positive constants independent of $h, \tau, n$ and $\varepsilon$.
\label{thm:TSFP}
\end{theorem}

\noindent
\emph{Proof.} Noticing that
\begin{equation}
I_M\Phi^n - \Phi(t_n) = I_M\Phi^n - P_M\Phi(t_n) + P_M\Phi(t_n) - \Phi(t_n),
\end{equation}
under the assumption (B), we get from the standard Fourier projection properties\cite{ST}
\begin{equation}
\left\|I_M\Phi^n - \Phi(t_n)\right\|_{L^2} \leq \left\|I_M\Phi^n - P_M\Phi(t_n)\right\|_{L^2} + C_2 h^m, \ 0 \leq n \leq \frac{T/\eps}{\tau}.
\end{equation}
Thus, it suffices to consider the error function ${\bf e}^n \in Y_M$ at $t_n$ as
\begin{equation}
{\bf e}^n := {\bf e}^n(x) = I_M\Phi^n - P_M\Phi(t_n),	\quad 0 \leq n \leq \frac{T/\eps}{\tau}.
\end{equation}
Since $\Phi^0_j = \Phi_0(x_j)$, it is obvious from (B) that $\|{\bf e}^0\|_{L^2} \leq C_3 h^m$.
From the local truncation error \eqref{eq:CalE}, we have the error equation for ${\bf e}^n$ ($0 \leq n \leq \frac{T/\varepsilon}{\tau}-1$),
\begin{align}
{\bf e}^{n+1} = I_M\Phi^{n+1} - P_M\mathcal{S}_{\tau}(\Phi(t_n)) + P_M(\mathcal{E}^{n}).
\label{eq:error_decomp}
\end{align}
Recall the semi-discretization \eqref{eq:S2} and the full-discretization \eqref{eq:TSFP}, we have
\begin{equation*}
\begin{split}
& I_M\Phi^{n+1} = e^{\frac{\tau {\bT}}{2}}(I_M\Phi^{(2)}), \ I_M\Phi^{(2)} = I_M\left(e^{\eps\tau\bV}\Phi^{(1)}\right), \ I_M\Phi^{(1)} = e^{\frac{\tau {\bT}}{2}}(I_M\Phi^n),\\
& P_M(\mathcal{S}_{\tau}(\Phi(t_n))) = e^{\frac{\tau {\bT}}{2}}(P_M\Phi^{\langle2\rangle}), \ \Phi^{\langle2\rangle} := e^{\eps\tau\bV}\Phi^{\langle1\rangle}, \ \Phi^{\langle1\rangle} := e^{\frac{\tau {\bT}}{2}}\Phi(t_n).
\end{split}	
\end{equation*}
In view of the facts that $I_M$ and $P_M$ are identical on $Y_M$ and $e^{\tau {\bT}/2}$ preserves the $H^k$-norm ($k \geq 0$), using the Taylor expansion $e^{\eps \tau {\bV}} = 1+\eps\tau{\bV}\int^1_0 e^{\eps\tau\theta {\bV}} d\theta$ and assumptions (A) and (B), we have
\begin{align}
&\left\|I_M\Phi^{n+1} - P_M(\mathcal{S}_{\tau}(\Phi(t_n)))\right\|_{L^2}	= \left\|I_M\Phi^{(2)} - P_M\Phi^{\langle2\rangle}\right\|_{L^2},\\
&\left\|P_M\Phi^{\langle2\rangle}- I_M\Phi^{\langle2\rangle}\right\|_{L^2} \leq C_4 \eps \tau h^m.
\end{align}
In addition, noticing $\bV = -i(V(x)I_2 - A_1(x)\sigma_1)$, by direct computation and Parseval's identity, we have
\begin{align}
\left\|I_M\Phi^{(2)}- I_M\Phi^{\langle2\rangle}\right\|_{L^2} & = \sqrt{h\sum_{j=0}^{M-1} \left\vert\Phi^{(2)}_j - \Phi^{\langle2\rangle}(x_j)\right\vert^2}= \sqrt{h\sum_{j=0}^{M-1} \left\vert\Phi^{(1)}_j - \Phi^{\langle1\rangle}(x_j)\right\vert^2} \nn \\
& = \left\|I_M \Phi^{(1)} - I_M \Phi^{\langle 1\rangle}\right\|_{L^2}= \left\|I_M\Phi^n - P_M\Phi(t_n)\right\|_{L^2} \nn\\
& = \left\|{\bf e}^n\right\|_{L^2}.
\end{align}
From \eqref{eq:E_local}, \eqref{eq:E_local2} and the assumption (B), it is clear that there exists $C_5>0$ such that $\|\mathcal{E}^n\|_{L^2}\leq C_5\eps\tau^3$ for $0 \leq n \leq \frac{T/\eps}{\tau} - 1$. Taking the $L^2$-norm on both sides of \eqref{eq:error_decomp} and combining the above estimates, we obtain for $0 \leq n \leq \frac{T/\eps}{\tau} - 1$,
\begin{align}
\left\|{\bf e}^{n+1}\right\|_{L^2} & \leq \left\|I_M\Phi^{(2)} - P_M\Phi^{\langle2\rangle}\right\|_{L^2} + \left\|\mathcal{E}^{n}\right\|_{L^2} \nn \\
& \leq \left\|I_M\Phi^{(2)}- I_M\Phi^{\langle2\rangle}\right\|_{L^2} +  \left\|P_M\Phi^{\langle2\rangle}- I_M\Phi^{\langle2\rangle}\right\|_{L^2}  + \left\|\mathcal{E}^{n}\right\|_{L^2} \nn \\
&\leq \left\|{\bf e}^n\right\|_{L^2} + C_6\left(\eps\tau h^m + \eps \tau^3\right),
\end{align}
where $C_6 = \max\{C_4, C_5\}$.
Thus, we arrive at
\begin{equation}
\left\|{\bf e}^{n+1}\right\|_{L^2} \leq C_6 \eps t_{n+1}\left(h^m + \tau^2\right) + C_3 h^m, \quad 	0 \leq n \leq \frac{T/\eps}{\tau} - 1,
\end{equation}
which completes the proof of Theorem \ref{thm:TSFP} by taking $C_0 = C_2 + C_3$ and $C_1 = C_6$.
\hfill $\square$

\section{Improved uniform error bounds} \label{sec:imprvbds}
In this section, we establish the improved uniform error bounds for the time-splitting methods applied to the Dirac equation \eqref{eq:Dirac_1D} up to the long-time $T_{\eps}$ under the assumptions (A) and (B).

\subsection{For semi-discretization}
\begin{theorem}
Let $\Phi^{[n]}(x)$ be the numerical approximation obtained from the semi-discretized second-order time-splitting \eqref{eq:S2} for the Dirac equation \eqref{eq:Dirac1d}. Under the assumptions (A) and (B), for $0 < \tau_0 < 1$ sufficiently small and independent of $\eps$, when $0 < \tau < \alpha \frac{\pi(b-a)\tau_0}{\sqrt{\tau_0^2(b-a)^2 + 4\pi^2(1+\tau_0)^2}}$ for a fixed constant $\alpha \in (0, 1)$, we have the following improved uniform error bound for any $\eps\in(0, 1]$
\begin{equation}
\|\Phi(t_n, x) - \Phi^{[n]}\|_{L^2} \lesssim \eps\tau^2 + \tau^m_0, \quad  0 \leq n \leq \frac{T/\varepsilon}{\tau}.
\label{eq:impr_semi}	
\end{equation}
In particular, if the exact solution is smooth, e.g., $\Phi(t, x) \in  L^\infty([0, T_\eps]; (H^{\infty}_{\rm per}(\Omega))^2)$, the last term $\tau_0^m$ decays exponentially fast and could be ignored practically for small enough $\tau_0$, where the improved uniform error bound for sufficiently small $\tau$ will be
\begin{equation}
\|\Phi(t_n, x) - \Phi^{[n]}\|_{L^2} \lesssim  \eps\tau^2, \quad  0 \leq n \leq \frac{T/\varepsilon}{\tau}.
\end{equation}
\label{thm:improved_semi}	
\end{theorem}

\begin{proof}
From the local truncation error \eqref{eq:CalE} and the error equation \eqref{eq:semi_grow}, we have
\begin{equation}
{\bf e}^{[n+1]} =  \mathcal{S}_{\tau}(\Phi^{[n]}) -  \mathcal{S}_{\tau}(\Phi(t_n)) +  \mathcal{E}^{n} = e^{\tau \bT} {\bf e}^{[n]} + W^n + \mathcal{E}^{n}, \quad n \geq 0,
\label{eq:en1}
\end{equation}
where $W^n := W^n(x)$ ($n = 0, 1, \ldots$) is given by
\begin{equation}
W^n(x) = \eps \tau 	e^{\frac{\tau \bT}{2}}{\bV}\int^1_0 e^{\eps\tau\theta \bV} d\theta \; e^{\frac{\tau \bT}{2}}  {\bf e}^{[n]}.
\end{equation}
Under the assumption (A), we have the following estimate
\begin{equation}
\left\|W^n(x)\right\|_{L^2} \lesssim \eps \tau \left\|{\bf e}^{[n]} \right\|_{L^2}.
\label{eq:Wb}
\end{equation}
Based on \eqref{eq:en1}, we arrive at
\begin{equation}
{\bf e}^{[n+1]} =  e^{(n+1)\tau \bT}	{\bf e}^{[0]} + \sum_{k=0}^{n} e^{(n-k)\tau \bT}\left(W^k + \mathcal{E}^k\right), \quad 0 \leq n \leq  \frac{T/\eps}{\tau}-1.
\end{equation}
Combining \eqref{eq:E_local}, \eqref{eq:E_local2} and \eqref{eq:Wb}, noticing ${\bf e}^{[0]} = {\bm 0}$, we have the estimates for $0 \leq n \leq  \frac{T/\eps}{\tau}-1$,
\begin{equation}
\left\|{\bf e}^{[n+1]}\right\|_{L^2}\lesssim \eps\tau^2 + \eps\tau \sum_{k=0}^{n} \left\|{\bf e}^{[k]}\right\|_{L^2} + \left\|\sum_{k=0}^{n} e^{(n-k)\tau \bT}\mathcal{F}(\Phi(t_k))\right\|_{L^2}.	
\label{eq:erfinial}
\end{equation}
In order to obtain the improved uniform error bounds, we shall employ the {\textbf{regularity compensation oscillation}} (RCO) technique\cite{BCF1,BCF2} to deal with the last term in the RHS of the inequality \eqref{eq:erfinial}.

The key idea is a summation-by-parts procedure combined with spectrum cut-off and phase cancellation. First, we employ the spectral projection on $\Phi(t_k)$ such that only finite Fourier modes of $\Phi(t_k)$ need to be considered and the projection error could be controlled by the regularity of $\Phi(t_k)$. Then, we apply the summation-by parts formula to the low Fourier modes such that the phase could be cancelled for small $\tau$ (the terms of the type $\sum_{k=0}^{n} e^{(n-k)\tau \bT}$) and an extra order of $\eps$ could be gained from the terms like $\mathcal{F}(\Phi(t_k)) - \mathcal{F}(\Phi(t_{k+1}))$.

From the Dirac equation \eqref{eq:Dirac1d} and the assumption (A), we notice that $\partial_t \Phi(t, x) - \bT \Phi(t, x) = O(\eps)$. In order to gain an extra order of $\eps$, it is natural to introduce the `twisted variable' as
\begin{equation}
\Psi(t, x) = e^{-t\bT} \Phi(t, x), \quad t \geq 0,	
\end{equation}
which satisfies the equation
\begin{equation}
\partial_t \Psi(t, x) = \eps e^{-t\bT}\left(\bV e^{t\bT}\Psi(t, x)\right), \quad t \geq 0.
\end{equation}
Noticing $ \bT = -\sigma_1\partial_x - i\sigma_3$, under the assumptions (A) and (B), we can prove that
\begin{equation}\label{eq:Psi_regu}
	\|\Psi\|_{L^{\infty}([0, T_\eps]; (H_{\rm per}^m(\Omega))^2)}\lesssim 1, \quad \|\partial_{t}\Psi\|_{L^{\infty}([0, T_\eps]; (H_{\rm per}^m(\Omega))^2)} \lesssim \eps
\end{equation}
and
\begin{equation}
\left\|\Psi(t_{n+1}) - \Psi(t_n)\right\|_{H^m}	\lesssim \eps \tau, \quad 0 \leq n \leq \frac{T/\eps}{\tau} - 1.
\label{eq:S1twist}
\end{equation}
The RCO technique will be used to force $\partial_t\Psi$ to appear with the gain of order $\eps$ for the summation-by-parts procedure in the last term $\sum_{k=0}^{n} e^{(n-k)\tau \bT}\mathcal{F}(\Phi(t_k))$, where the small $\tau$ is required to control the accumulation of the frequency of the type $ e^{(n-k)\tau \bT}$.

\noindent
{\textbf{Step 1.}} As introduced in Ref.~\refcite{BCF1}, we choose the cut-off parameter $\tau_0 \in (0, 1)$ and $M_0 = 2\lceil 1/\tau_0 \rceil \in \mathbb{Z}^+$ ($\lceil \cdot \rceil$ is the ceiling function) with $1/\tau_0 \leq M_0/2 < 1+ 1/\tau_0$. Under the assumptions (A) and (B), we have the following estimate
\begin{equation}
\left\|P_{M_0}\mathcal{F}(P_{M_0}\Phi(t_n)) - \mathcal{F}(\Phi(t_n))\right\|_{L^2}	\lesssim \eps \tau \tau_0^m.
\end{equation}
Based on the above estimates, \eqref{eq:erfinial} would imply for $0 \leq n \leq T_\eps/\tau-1$,
\begin{equation}
\left\|{\bf e}^{[n+1]}\right\|_{L^2}\lesssim \tau_0^m + \eps\tau^2 + \eps\tau \sum_{k=0}^{n} \left\|{\bf e}^{[k]}\right\|_{L^2} + \left\| \mathcal{L}^n\right\|_{L^2},
\label{eq:matl}
\end{equation}
where
\begin{equation}
\mathcal{L}^n = \sum_{k=0}^{n} e^{-(k+1)\tau \bT}P_{M_0}\mathcal{F}\left(e^{t_k\bT}(P_{M_0}\Psi(t_k))\right).
\end{equation}

\noindent
{\textbf{Step 2.}} Now, we concentrate on $\mathcal{L}^n$, which represents the summation of low Fourier modes. For $l \in \mathcal{T}_{M_0}$, define the index set $\mathcal{I}_l^{M_0}$ associated to $l$ as
\begin{equation}
\mathcal{I}_l^{M_0} = \left\{(l_1,l_2) \ \vert \ l_1+l_2 = l,\ l_1\in\mathbb{Z},\, l_2 \in\mathcal{T}_{M_0}\right\}.
\end{equation}
Following the notations in \eqref{eq:Qdef}, we denote
\begin{equation}
\Pi_l^+= Q_l\,\text{diag}(1,0) (Q_l)^T,\quad \Pi_l^- = Q_l\,\text{diag}(0,1) (Q_l)^T,
\end{equation}
where $\Pi_l^{\pm}$ are the projectors onto the eigenspaces of $\Gamma_l$ corresponding to the eigenvalues $\pm\delta_l$, respectively. Moreover, we have $(\Pi_l^\pm)^T=\Pi_l^\pm$, $\Pi_l^++\Pi_l^-=I_2$, $(\Pi_l^\pm)^2=\Pi_l^\pm$, $\Pi_l^{\pm}\Pi_l^{\mp}={\bf 0}$.
By direct computation, we  have
\begin{equation}
e^{t{\bm T}}P_{M_0}\Psi(t_k)=\sum\limits_{l\in\mathcal{T}_{M_0}}\left(e^{-it\delta_l}\Pi_l^++e^{it\delta_l}\Pi_l^-\right)\widehat{\Psi}_l(t_k)e^{i\mu_l(x-a)}.
\end{equation}
According to the definition of $\mathcal{F}$ in \eqref{eq:malF}, the expansion below follows
\begin{align}
&e^{-(k+1)\tau{\bm T}}P_{M_0}\left(e^{(\tau-s)\bT}\bV e^{(t_k+s)\bT}P_{M_0}\Psi(t_k))\right)\nonumber\\
&=\sum\limits_{l\in\mathcal{T}_{M_0}}\sum\limits_{(l_1,l_2)\in\mathcal{I}_l^{M_0}}\sum\limits_{\nu_j=\pm, j=1,2}\mathcal{G}_{k,l,l_1,l_2}^{\nu_1,\nu_2}(s) e^{i\mu_l(x-a)},
\end{align}
where  $\mathcal{G}_{k,l,l_1,l_2}^{\nu_1,\nu_2}(s)$ is a function of $s$ defined as
\begin{equation}
\mathcal{G}_{k,l,l_1,l_2}^{\nu_1, \nu_2}(s) =
e^{i(t_k+s)\delta^{\nu_1, \nu_2}_{l,l_2}}\Pi_l^{\nu_1}\widehat{\bV}_{l_1}\Pi_{l_2}^{\nu_2} \widehat{\Psi}_{l_2}(t_k),
\label{eq:mGdef}
\end{equation}
with $\delta_{l,l_2}^{\nu_1,\nu_2}=\nu_1 \delta_l-\nu_2 \delta_{l_2}$.
Then, the remainder term $\mathcal{L}^n$ in \eqref{eq:matl} reads
\begin{equation}
\label{eq:reminder-dec}
\mathcal{L}^n = \eps \sum\limits_{k=0}^n
\sum\limits_{l\in\mathcal{T}_{M_0}}\sum\limits_{(l_1,l_2)\in\mathcal{I}_l^{M_0}}\sum\limits_{ \nu_j=\pm, j=1,2} \lambda_{k, l, l_1, l_2}^{\nu_1, \nu_2} e^{i\mu_l(x-a)},
\end{equation}
where the coefficients  $\lambda_{k, l, l_1, l_2}^{\nu_1, \nu_2}$ given by
\begin{align}
\lambda_{k, l, l_1, l_2}^{\nu_1, \nu_2} & = \tau\mathcal{G}_{k,l,l_1,l_2}^{\nu_1, \nu_2}(\tau/2) - \int_0^\tau\mathcal{G}_{k,l,l_1,l_2}^{\nu_1, \nu_2}(s)\,ds \nn \\
&   = r_{l,l_2}^{\nu_1, \nu_2}e^{it_k\delta^{\nu_1, \nu_2}_{l,l_2}}c_{k,l,l_1,l_2}^{\nu_1, \nu_2},
\label{eq:Lambdak}
\end{align}
and
\begin{align}
& c_{k,l,l_1,l_2}^{\nu_1, \nu_2} = 	\Pi_l^{\nu_1}\widehat{\bV}_{l_1}\Pi_{l_2}^{\nu_2} \widehat{\Psi}_{l_2}(t_k), \\
& r_{l,l_2}^{\nu_1, \nu_2} = \tau e^{i\tau\delta^{\nu_1, \nu_2}_{l,l_2}/2} - \int^{\tau}_0 e^{is\delta^{\nu_1, \nu_2}_{l,l_2}}\;ds = O(\tau^3(\delta^{\nu_1, \nu_2}_{l,l_2})^2).
\label{eq:rest1}
\end{align}
We only need to consider the case $\delta^{\nu_1,\nu_2}_{l, l_2} \neq0$ as $ r_{l, l_2}^{\nu_1, \nu_2} = 0$ if $\delta^{\nu_1, \nu_2}_{l, l_2} = 0$. For $l\in\mathcal{T}_{M_0}$ and $(l_1, l_2)\in\mathcal{I}_l^{M_0}$, we have
\begin{equation}
\vert\delta^{\nu_1,\nu_2}_{l, l_2}\vert \leq 2\delta_{M_0/2}= 2\sqrt{1+\mu_{{M_0}/2}^2} < 2\sqrt{1 + \frac{4\pi^2 (1+\tau_0)^2}{\tau_0^2(b-a)^2}},
\end{equation}
which implies for $0< \tau \leq \alpha \frac{\pi(b-a)\tau_0}{\sqrt{\tau_0^2(b-a)^2 + 4\pi^2(1+\tau_0)^2}}$ with $0 < \alpha, \tau_0 < 1$, there holds $\frac{\tau}{2}\vert \delta_{l, l_2}^{\nu_1, \nu_2} \vert \leq \alpha\pi$.
Denoting $S^{\nu_1, \nu_2}_{n, l, l_2}=\sum\limits_{k=0}^n e^{it_k \delta_{l, l_2}^{\nu_1, \nu_2}}$ ($n\ge0$), for $0< \tau \leq \alpha \frac{\pi(b-a)\tau_0}{\sqrt{\tau_0^2(b-a)^2 + 4\pi^2(1+\tau_0)^2}}$, we obtain
\begin{equation}
\label{eq:Sbd1}
 \vert S^{\nu_1, \nu_2}_{n, l, l_2}\vert \leq \frac{1}{\vert \sin(\tau \delta_{l, l_2}^{\nu_1, \nu_2}/2)\vert}\leq\frac{C}{\tau\vert\delta_{l, l_2}^{\nu_1, \nu_2}\vert},\quad C = \frac{2\alpha\pi}{\sin(\alpha\pi)},\quad  \forall n\ge0.
\end{equation}
Using the summation-by-parts, we find that
\begin{align}
\sum_{k=0}^n\lambda_{k,l,l_1,l_2}^{\nu_1, \nu_2}= & \
r_{l, l_2}^{\nu_1, \nu_2}\left[\sum_{k=0}^{n-1}S^{\nu_1,\nu_2}_{k,l,l_2} (c_{k,l,l_1,l_2}^{\nu_1, \nu_2} - c_{k+1,l,l_1,l_2}^{\nu_1, \nu_2})  +S_{n, l, l_2}^{\nu_1, \nu_2}c_{n,l,l_1,l_2}^{\nu_1, \nu_2}\right],
\label{eq:lambdasum1}
\end{align}
with
\begin{equation}
c_{k,l,l_1,l_2}^{\nu_1, \nu_2}-c_{k+1,l,l_1,l_2}^{\nu_1, \nu_2} = 	\Pi_l^{\nu_1}\widehat{\bV}_{l_1}\Pi_{l_2}^{\nu_2}\left(\widehat{\Psi}_{l_2}(t_k) - \widehat{\Psi}_{l_2}(t_{k+1})\right).
\label{eq:cksum1}
\end{equation}
Combining \eqref{eq:rest1}, \eqref{eq:Sbd1}, \eqref{eq:lambdasum1}, and \eqref{eq:cksum1}, we have
\begin{equation}
\left\vert\sum_{k=0}^n \lambda_{k,l,l_1,l_2}^{\nu_1, \nu_2}\right\vert
\lesssim   \tau^2\vert\delta_{l, l_2}^{\nu_1, \nu_2}\vert\left\vert\widehat{\bV}_{l_1}\right\vert\left[\sum\limits_{k=0}^{n-1} \left\vert\widehat{\Psi}_{l_2}(t_k)-\widehat{\Psi}_{l_2}(t_{k+1})\right\vert
 + \left\vert\widehat{\Psi}_{l_2}(t_n)\right\vert\right] .\label{eq:sumlambda1}
\end{equation}

\noindent
{\textbf{Step 3.}} Now, we are ready to give the improved estimates. For $l\in\mathcal{T}_{M_0}$ and $(l_1, l_2)\in\mathcal{I}_l^{M_0}$, simple calculations show ($l = l_1 + l_2$)
\begin{equation}
\vert \delta_{l, l_2}^{\nu_1, \nu_2}\vert \lesssim\prod_{j=1}^2(1+\mu_{l_j}^2)^{1/2}.
\label{eq:mlbd1}
\end{equation}
Based on \eqref{eq:reminder-dec}, \eqref{eq:sumlambda1} and \eqref{eq:mlbd1},  using Cauchy inequality, we have
\begin{align}
&\left\| \mathcal{L}^n\right\|^2_{L^2}  \label{eq:sumlambda-21}\\
&= \eps^2
\sum\limits_{l\in\mathcal{T}_{N_0}}\big\vert \sum\limits_{(l_1,l_2)\in\mathcal{I}_l^{N_0}}\sum\limits_{\nu_j = \pm,  j=1, 2}\sum\limits_{k=0}^n\lambda_{k, l, l_1, l_2}^{\nu_1,\nu_2}\big\vert^2\nonumber\\
&\lesssim \varepsilon^2 \tau^4
\bigg\{\sum_{l\in\mathcal{T}_{M_0}}\bigg(\sum\limits_{(l_1, l_2)\in\mathcal{I}_l^{N_0}}\left\vert \widehat{\bV}_{l_1}\right\vert \left\vert \widehat{\Psi}_{l_2}(t_n)\right\vert \prod_{j=1}^2(1+\mu_{l_j}^2)^{1/2} \bigg)^{2}\nonumber\\
&\quad+ n \sum\limits_{k=0}^{n-1} \bigg[\sum_{l\in\mathcal{T}_{M_0}}\bigg(\sum\limits_{(l_1, l_2)\in\mathcal{I}_l^{N_0}}
\left\vert\widehat{\bV}_{l_1}\right\vert\left\vert\widehat{\Psi}_{l_2}(t_k)-\widehat{\Psi}_{l_2}(t_{k+1})\right\vert\prod_{j=1}^2(1+\mu_{l_j}^2)^{1/2} \bigg)^2\bigg]\bigg\}.\nonumber
\end{align}
In order to estimate each term in the above inequality, we use the auxiliary function $\Theta(x)=\sum_{l\in\mathbb{Z}}(1+\mu_l^2)^{1/2}\left\vert\widehat{\Psi}_l(t_n)\right\vert e^{i\mu_l(x-a)}$, where $\Theta(x)\in \left(H_{\rm per}^{m-1}(\Omega)\right)^2$ which can be verified from the assumption (B) and we can prove $\|\Theta(x)\|_{H^{s}}\lesssim\|\Psi(t_n)\|_{H^{s+1}}$ ($s \leq m-1$). Similarly, introduce the function $\bU(x) = \sum_{l\in\mathbb{Z}}(1+\mu_l^2)^{1/2}\left\vert \widehat{\bV}_l\right\vert e^{i\mu_l(x-a)}$, where $\bU(x) \in {W^{m-1, \infty}_{\rm per}} (\Omega)$ as can be derived from the assumption (A).  We can prove directly that  $\|\bU(x)\|_{W^{0,\infty}}\lesssim\|\bV(x)\|_{W^{2, \infty}}$. Expanding
\[\bU(x)\Theta(x)=\sum\limits_{l\in\mathbb{Z}}\sum\limits_{l_1+l_2=l} \prod_{j=1}^2 (1+\mu_{l_j}^2)^{1/2} \left\vert\widehat{\bV}_{l_1}\right\vert\left\vert\widehat{\Psi}_{l_2}(t_n)\right\vert e^{i\mu_l(x-a)},\]
we can obtain
\begin{align}
&\sum_{l\in\mathcal{T}_{M_0}} \bigg(\sum\limits_{(l_1, l_2)\in\mathcal{I}_l^{M_0}}\left\vert\widehat{\bV}_{l_1}\right\vert\left\vert\widehat{\Psi}_{l_2}(t_n)\right\vert \prod_{j=1}^2(1+\mu_{l_j}^2)^{1/2} \bigg)^2\nn\\
& \leq \ \|\bU(x)\Theta(x)\|^2_{L^2} \lesssim \|\bV(x)\|_{W^{2, \infty}}^2 \|\Psi(t_n)\|_{H^1}^2,
\label{eq:est_ll}
\end{align}
which together with the assumption (A), \eqref{eq:Psi_regu} and \eqref{eq:S1twist} gives
\begin{align}
\left\| \mathcal{L}^n\right\|^2_{L^2} & \lesssim \varepsilon^2\tau^4\|\bV(x)\|_{W^{2, \infty}}^2 \bigg(\|\Psi(t_n)\|_{H^1}^2 + n\sum\limits_{k=0}^{n-1}
\|\Psi(t_k) - \Psi(t_{k+1})\|_{H^1}^2\bigg)\nn \\
& \lesssim \varepsilon^2 \tau^4+ n^2\varepsilon^4 \tau^6
\lesssim \varepsilon^2\tau^4.
\label{eq:est-l2l}
\end{align}
for $0 \leq n \leq \frac{T/\varepsilon}{\tau}-1$, where the same trick is applied to the rest terms. Combining \eqref{eq:matl} and \eqref{eq:est-l2l}, we have
\begin{equation}
\|{\bf e}^{n+1}\|_{L^2} \lesssim  \tau_0^m + \eps \tau^2 +\varepsilon\tau \sum\limits_{k=0}^n\|{\bf e}^{k}\|_{L^2},\quad 0\leq n\leq\frac{T/\eps}{\tau}-1.
\end{equation}
Discrete Gronwall's inequality would yield $\|{\bf e}^{n+1}\|_{L^2} \lesssim \eps \tau^2 + \tau_0^m$ $ (0\leq n\leq\frac{T/\varepsilon}{\tau}-1)$, and the proof for the improved uniform error bound \eqref{eq:impr_semi} in Theorem \ref{thm:improved_semi} is completed.
\end{proof}

\begin{remark}
$\tau_0 \in (0, 1)$ is a parameter introduced in analysis and the requirement on $\tau$ (essentially $\tau\lesssim\tau_0$) enables the improved estimates on the low Fourier modes $|l|\leq1/\tau_0$, where the error constant depends on $\alpha$. $\tau_0$ can be arbitrary as long as the assumed relation between $\tau$ and $\tau_0$ holds, i.e. $\tau_0$ could be either fixed, or dependent on $\tau$,  e.g. $\tau_0=\frac{\sqrt{16\pi^2+(b-a)^2}}{\alpha(b-a)\pi}\tau$.
\end{remark}

\begin{remark} The improved uniform error bounds are established for the second-order Strang splitting method. Under appropriate assumptions of the exact solution, the improved uniform error bounds can be extended to the first-order Lie-Trotter splitting and the fourth-order PRK splitting method with improved uniform error bounds at $\eps \tau$ and $\eps \tau^{4}$, respectively.
\end{remark}

\subsection{For full-discretization}
For the TSFP method \eqref{eq:TSFP}, we could establish the following improved uniform error bounds.
\begin{theorem}
Let $\Phi^n$ be the approximation obtained from the TSFP \eqref{eq:TSFP} for the Dirac equation \eqref{eq:Dirac1d}. Under the assumptions (A) and (B), for $0 < \tau_0 < 1$ sufficiently small and independent of $\eps$, when $0 < \tau \leq \alpha \frac{\pi(b-a)\tau_0}{\sqrt{\tau_0^2(b-a)^2 + 4\pi^2(1+\tau_0)^2}}$ for a fixed constant $\alpha \in (0, 1)$, we have
\begin{equation}
\left\|\Phi(t_n, x) - I_M\Phi^n\right\|_{L^2} \lesssim  h^m + \eps\tau^2 + \tau^m_0, \quad  0 \leq n \leq \frac{T/\varepsilon}{\tau}
\label{eq:improved_full}	
\end{equation}
for any $0 < \varepsilon \leq 1$. In particular, if the exact solution is smooth, e.g., $\Phi(t, x) \in L^\infty([0, T_\eps]; (H^{\infty}_{\rm per}(\Omega))^2)$, the improved uniform error bounds for sufficiently small $\tau$ will be
\begin{equation}
\left\|\Phi(t_n, x) - I_M\Phi^n\right\|_{H^1} \lesssim  h^m + \eps\tau^2, \quad  0 \leq n \leq \frac{T/\varepsilon}{\tau}.
\end{equation}
\label{thm:improved_full}
\end{theorem}

\begin{proof}
From the error estimates in previous sections, we have for $0 \leq n\leq \frac{T/\eps}{\tau}$,
\begin{equation}
\left\|\Phi(t_n, x) - \Phi^{[n]}\right\|_{L^2} \lesssim \eps\tau^2 + \tau^m_0,\quad	\left\|\Phi^{[n]}-P_M\Phi^{[n]}\right\|_{L^2} \lesssim h^m.
\end{equation}
Since $\Phi(t_n, x) - I_M\Phi^n = \Phi(t_n, x) - \Phi^{[n]}+ \Phi^{[n]} - P_M\Phi^{[n]} + P_M\Phi^{[n]} - I_M\Phi^n$, we derive
\begin{equation}
\left\|\Phi(t_n, x) - I_M\Phi^n\right\|_{L^2} \leq \left\|P_M\Phi^{[n]} - I_M\Phi^n\right\|_{L^2} + C_1\left(h^m + \eps\tau^2 + \tau^m_0 \right),
\label{eq:diff}
\end{equation}
where $C_1$ is a constant independent of $h$, $\tau$, $n$, $\eps$ and $\tau_0$. As a result, it remains to establish the estimates  on the error function ${\widetilde{\bf e}}^n := {\widetilde{\bf e}}^n(x) \in Y_M$  given as
\[
{\widetilde{\bf e}}^n(x) := P_M\Phi^{[n]} - I_M\Phi^n,\quad 0 \leq n \leq \frac{T/\eps}{\tau}.\]
From \eqref{eq:TSFP} and \eqref{eq:S2}, we get
\begin{align*}
& I_M\Phi^{n+1} = e^{\frac{\tau {\bT}}{2}}\left(I_M\left(e^{\eps\tau\bV}e^{\frac{\tau {\bT}}{2}}(I_M\Phi^n)\right)\right),\\
& P_M\Phi^{[n+1]} =  e^{\frac{\tau {\bT}}{2}}\left(P_M\left(e^{\eps\tau\bV} e^{\frac{\tau {\bT}}{2}}(P_M\Phi^{[n]})\right)\right),
\end{align*}
which lead to
\begin{align}\label{eq:errg:f}
\widetilde{\bf e}^{n+1}=e^{\tau\bT}\widetilde{\bf e}^{n}+W^n(x),
\end{align}
where
\begin{equation*}
W^n(x) = e^{\frac{\tau\bT}{2}}\left[P_M\left((e^{\eps\tau\bV}-1)e^{\frac{\tau {\bT}}{2}}(P_M\Phi^{[n]})\right) - I_M\left((e^{\eps\tau\bV}-1)e^{\frac{\tau {\bT}}{2}}(I_M\Phi^n)\right)\right].
\end{equation*}
Similar to the error estimates in Ref.~\refcite{BCJY}, we have the following error bounds
\begin{equation}
\left\|W^n(x)\right\|_{L^2} \lesssim \eps\tau\left(h^m + \left\|\widetilde{\bf e}^n\right\|_{L^2}\right),	
\end{equation}
Thus, we could obtain
\begin{equation}
\left\|\widetilde{\bf e}^{n+1}\right\|_{L^2} \leq \left\|\widetilde{\bf e}^{n}\right\|_{L^2} +  C_2 \eps\tau\left(h^m + \left\|\widetilde{\bf e}^n\right\|_{L^2}\right), \quad 0 \leq n \leq \frac{T/\eps}{\tau}-1,
\end{equation}
where $C_2$ is a constant independent of $h$, $\tau$, $n$ and $\eps$. Since $\widetilde{\bf e}^0 = P_M \Phi_0 - I_M \Phi_0$, we have $\left\|\widetilde{\bf e}^0\right\|_{L^2} \lesssim h^m$ and discrete Gronwall's inequality implies $\left\|\widetilde{\bf e}^{n+1}\right\|_{L^2} \lesssim h^m$ for $0 \leq n \leq \frac{T/\eps}{\tau}-1$. Combining the above estimates with \eqref{eq:diff}, we derive
\[\left\|\Phi(t_n, x) - I_M\Phi^n\right\|_{L^2} \lesssim h^m + \eps \tau^2+\tau^m_0,\quad 0\leq n\leq \frac{T/\eps}{\tau},\]
which shows the improved uniform error bound \eqref{eq:improved_full} and the proof for Theorem \ref{thm:improved_full} is completed.
\end{proof}

\section{Numerical results}
\label{sec:num}
In this section, we present numerical results of the TSFP method for the long-time dynamics of the Dirac equation with $O(\eps)$-potentials up to the long-time $T_\eps=T/\eps$.

\subsection{For $\varepsilon = 1$ with $T  \gg 1$ regime}
First, we show an example to confirm that the uniform error bound grows linearly  with respect to the time $t$. We take $\Omega = (0, 1)$, the electromagnetic potentials
\begin{equation}
V(x) = x^2(x-1)^2+1, \quad A_1(x)= x^2(x-1)^2+1, \quad x \in \Omega.
\label{eq:potential}
\end{equation}
and the $(H^2_{\rm per}(\Omega))^2$ initial data
\begin{equation}
\phi_1(x) = \phi_2(x) = \frac{1}{2}x^2(1-x)^2 +3, \quad x \in \Omega.
\label{eq:num_in}
\end{equation}
The regularity is enough to ensure the uniform and improved error bounds. The `exact' solution $\Phi(t, x)$ is obtained numerically by the TSFP \eqref{eq:TSFP} with a very fine mesh size $h_e = 1/128$ and  time step size $\tau_e = 10^{-4}$. To quantify the error, we introduce the following error functions:
\begin{equation*}
e(t_n) = \left\|\Phi(t_n, x) - I_N \Phi^n\right\|_{L^2}, \quad e_{\max}(t_n) = \max_{0 \leq q \leq n}e(t_q).	
\end{equation*}
In the rest of the paper, the spatial mesh size is always chosen sufficiently small and thus spatial errors can be ignored when considering the long-time error growth and/or the temporal errors.

\begin{figure}[ht!]
\centerline{\includegraphics[width=12cm,height=5.5cm]{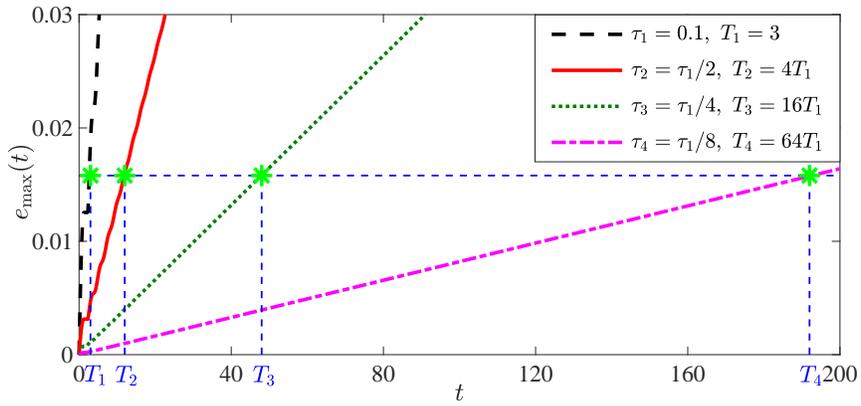}}
\caption{Long-time temporal errors of the TSFP \eqref{eq:TSFP} for the Dirac equation \eqref{eq:Dirac_1D} with $\varepsilon = 1$ and different time step $\tau$.}
\label{fig:long_tau}
\end{figure}


\begin{figure}[ht!]
\centerline{\includegraphics[width=12cm,height=5.5cm]{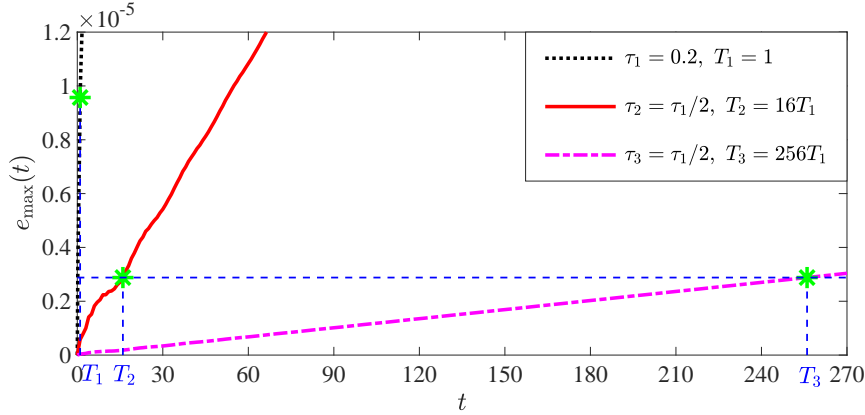}}
\caption{Long-time temporal errors of the PRK4 method for the Dirac equation \eqref{eq:Dirac_1D} with $\varepsilon = 1$ and different time step $\tau$.}
\label{fig:long_PRK}
\end{figure}

Fig. \ref{fig:long_tau} depicts the long-time temporal errors of the TSFP method for the Dirac equation \eqref{eq:Dirac_1D} with $\varepsilon = 1$ and different time step $\tau$, which shows that the uniform errors grows linearly  with respect to the time. In addition, for a given accuracy bound, the time to exceed  the error bar is quadruple when the time step is half, which also confirms the linear growth. For comparisons, Fig. \ref{fig:long_PRK} plots the long-time errors of the PRK4 method, which indicates that higher order time-splitting methods could get better accuracy with the same time step  size as well as longer time simulations within a given accuracy bound.

\subsection{For $\varepsilon\to0$ with fixed $T$ regime}
Next, we report the convergence test for the TSFP method \eqref{eq:TSFP} for the Dirac equation \eqref{eq:Dirac_1D} with the electromagnetic potentials \eqref{eq:potential} and the initial data \eqref{eq:num_in}.

\begin{figure}[ht!]
\centerline{\includegraphics[width=12cm,height=5.5cm]{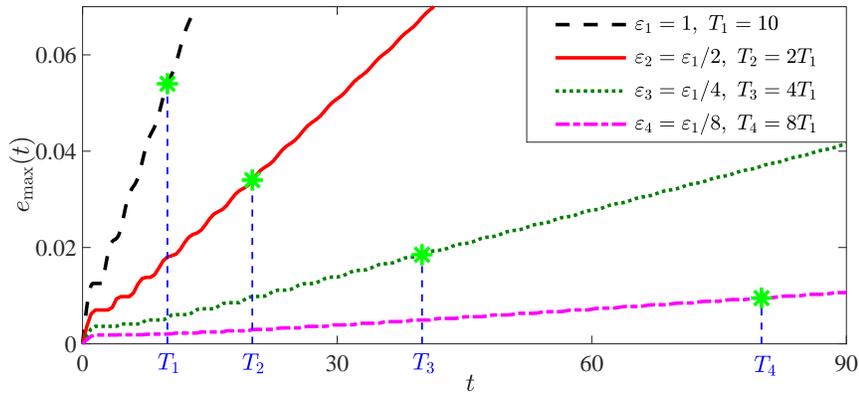}}
\caption{Long-time temporal errors of the TSFP \eqref{eq:TSFP} for the Dirac equation \eqref{eq:Dirac_1D} with $\tau = 0.1$ and different $\eps$.}
\label{fig:long_eps}
\end{figure}

\begin{figure}[ht!]
\begin{minipage}{0.49\textwidth}
\centerline{\includegraphics[width=6.3cm,height=5.5cm]{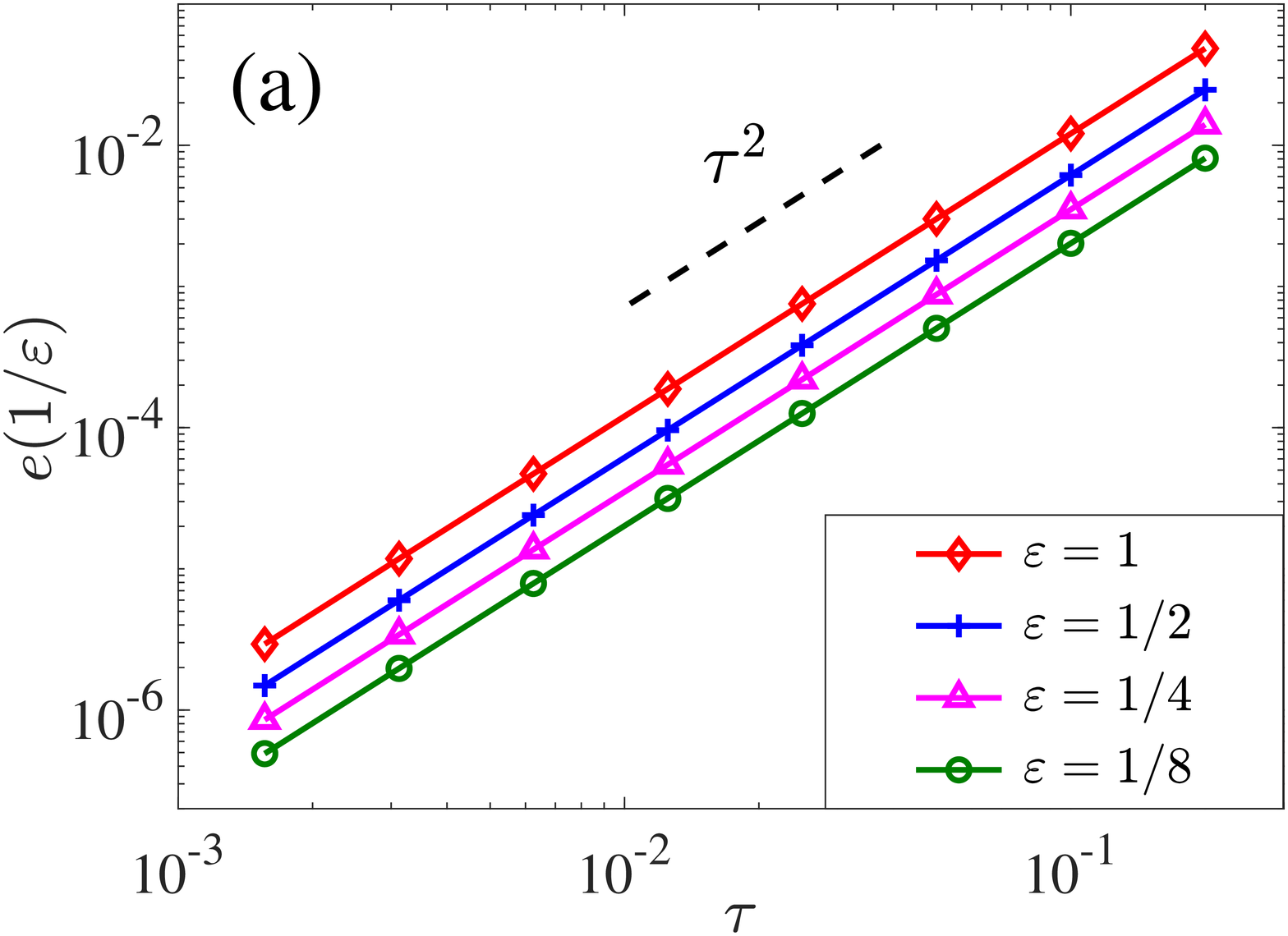}}
\end{minipage}
\begin{minipage}{0.49\textwidth}
\centerline{\includegraphics[width=6.3cm,height=5.5cm]{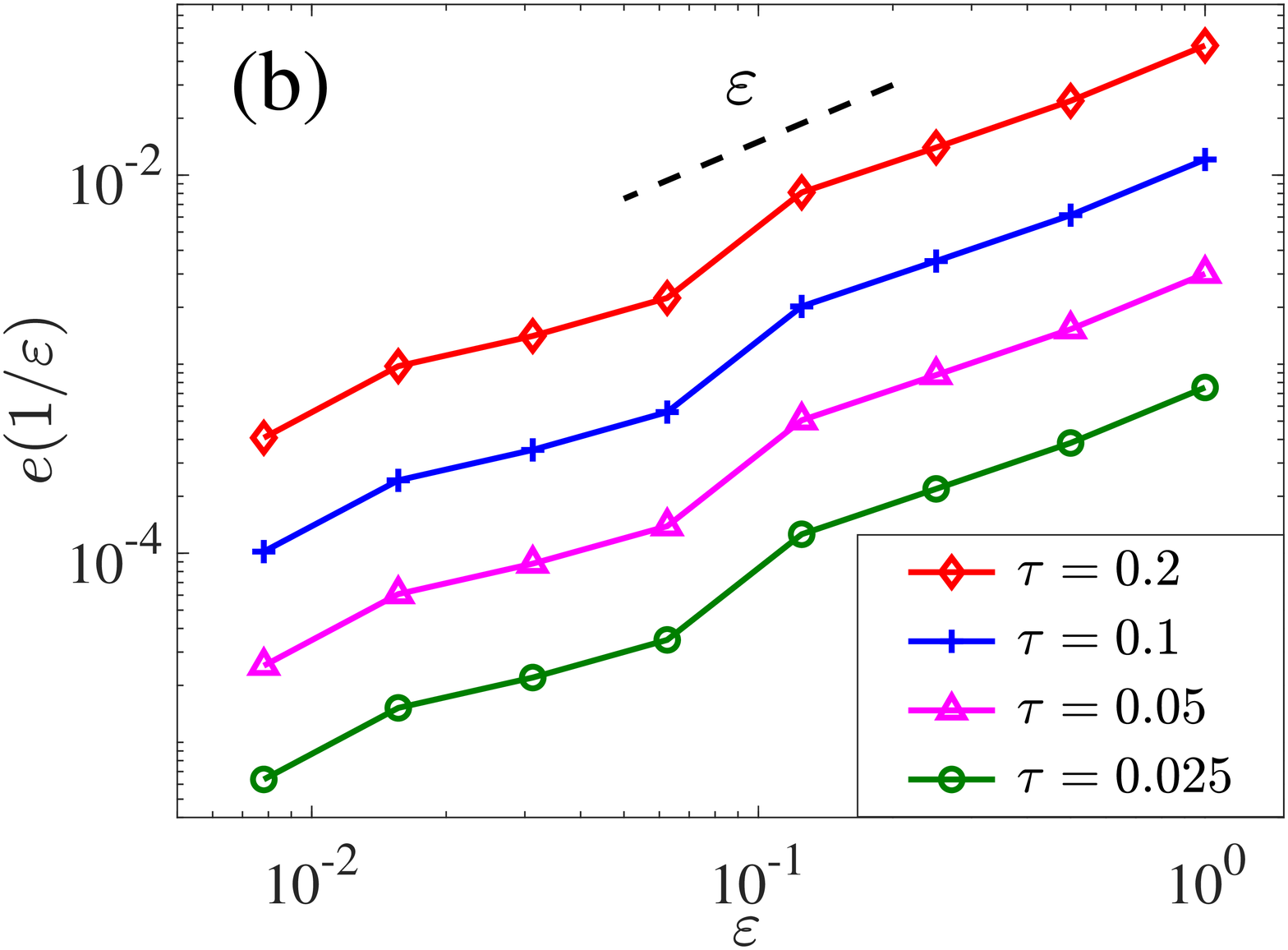}}
\end{minipage}
\caption{Temporal convergence rates of the TSFP \eqref{eq:TSFP} for the Dirac equation \eqref{eq:Dirac_1D} over long-time dynamics at $t = 1/\varepsilon$:
convergence rates in $\tau$ (a), and convergence rates in $\eps$ (b).}
\label{fig:temporal}
\end{figure}

\begin{figure}[ht!]
\begin{minipage}{0.49\textwidth}
\centerline{\includegraphics[width=6.3cm,height=5.5cm]{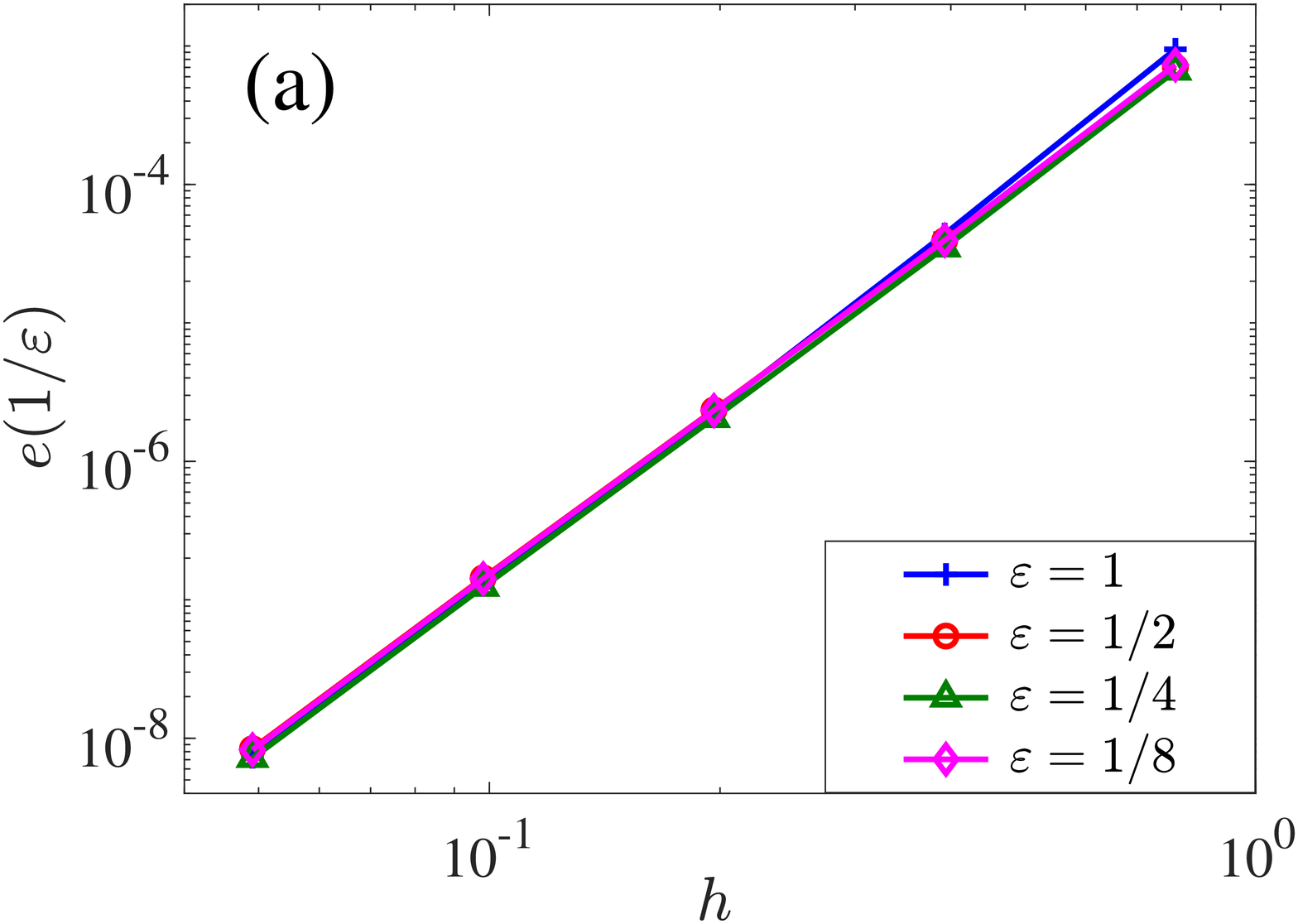}}
\end{minipage}
\begin{minipage}{0.49\textwidth}
\centerline{\includegraphics[width=6.3cm,height=5.5cm]{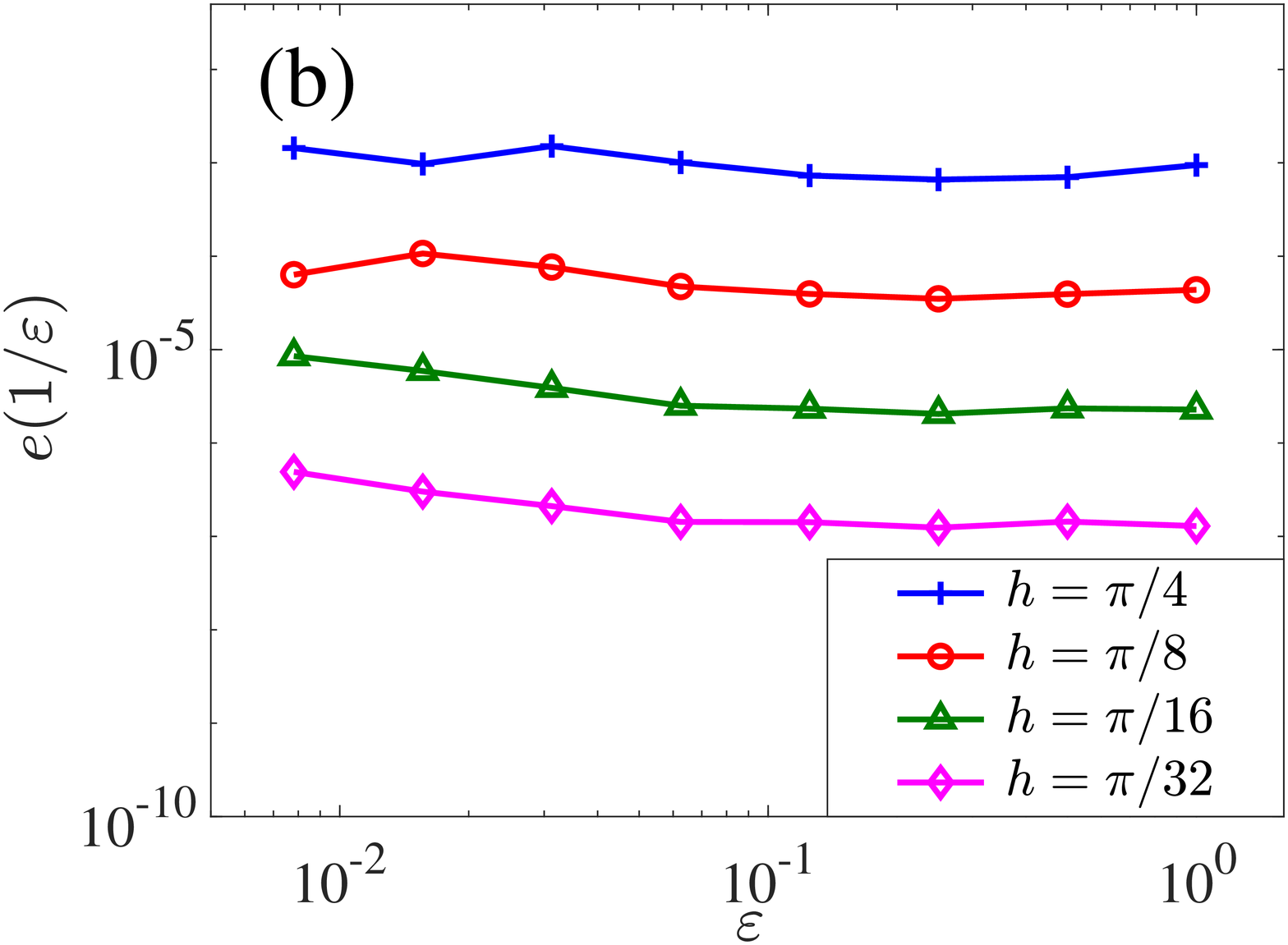}}
\end{minipage}
\caption{Spatial convergence rates of the TSFP \eqref{eq:TSFP} for the Dirac equation \eqref{eq:Dirac_1D} over long-time dynamics at $t = 1/\varepsilon$:
convergence rates in $\tau$ (a), and convergence rates in $\eps$ (b).}
\label{fig:spatial}
\end{figure}

Fig. \ref{fig:long_eps} plots the long-time errors of the TSFP method for the Dirac equation \eqref{eq:Dirac_1D} with the fixed time step $\tau$ and different $\varepsilon$, which confirms the improved uniform error bound at $O(\eps \tau^2$) up to the long-time at $O(1/\eps)$. Figs. \ref{fig:temporal} \&  \ref{fig:spatial} exhibit the temporal and spatial errors of the TSFP \eqref{eq:TSFP} for the Dirac equation \eqref{eq:Dirac_1D}  at $t= 1/\varepsilon$. Fig. \ref{fig:temporal}(a) shows the second-order convergence of the TSFP method in time. Each line in Fig. \ref{fig:temporal}(b) gives the global errors with a fixed time step $\tau$ and verifies that the global error performs like $O(\varepsilon\tau^2)$ up to the long-time at $O(1/\varepsilon)$. Each line in Fig. \ref{fig:spatial}(a) shows the spectral accuracy of the TSFP method in space and Fig. \ref{fig:spatial}(b) verifies the spatial errors are independent of the small parameter $\varepsilon$ in the long-time regime.

\subsection{Comparisons of different temporal discretizaitons}
In this subsection, we compare the long-time temporal errors of the time-splitting methods with the finite difference method (FDM) and the exponential wave integrator (EWI) method\cite{FY,FY2}. In order to compare the temporal errors, we adopt the Fourier pseudospectral method in space combined with each temporal discretization and choose a fine mesh size such that the spatial errors are neglected.

\begin{figure}[ht!]
\begin{minipage}{0.49\textwidth}
\centerline{\includegraphics[width=6.3cm,height=5.5cm]{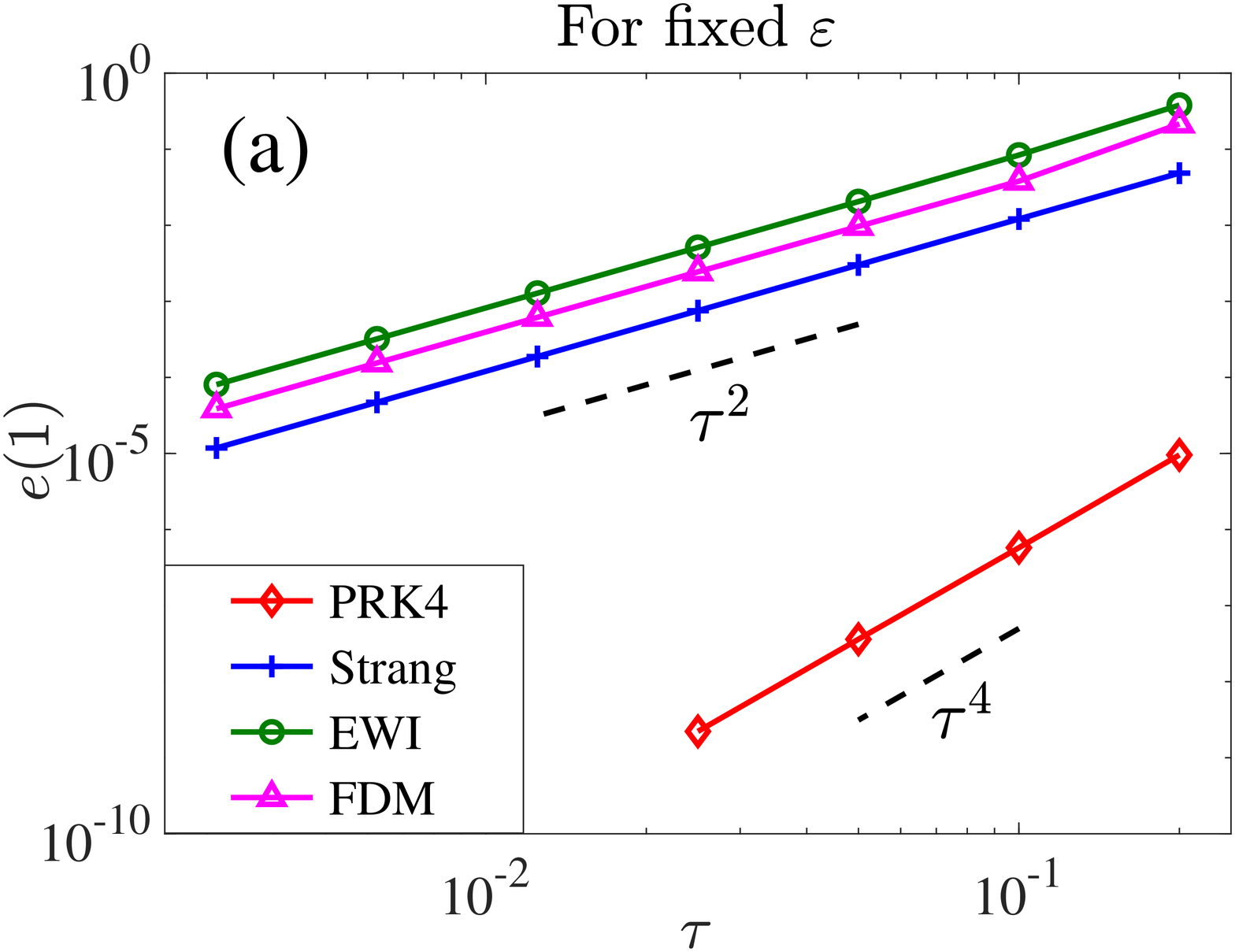}}
\end{minipage}
\begin{minipage}{0.49\textwidth}
\centerline{\includegraphics[width=6.3cm,height=5.5cm]{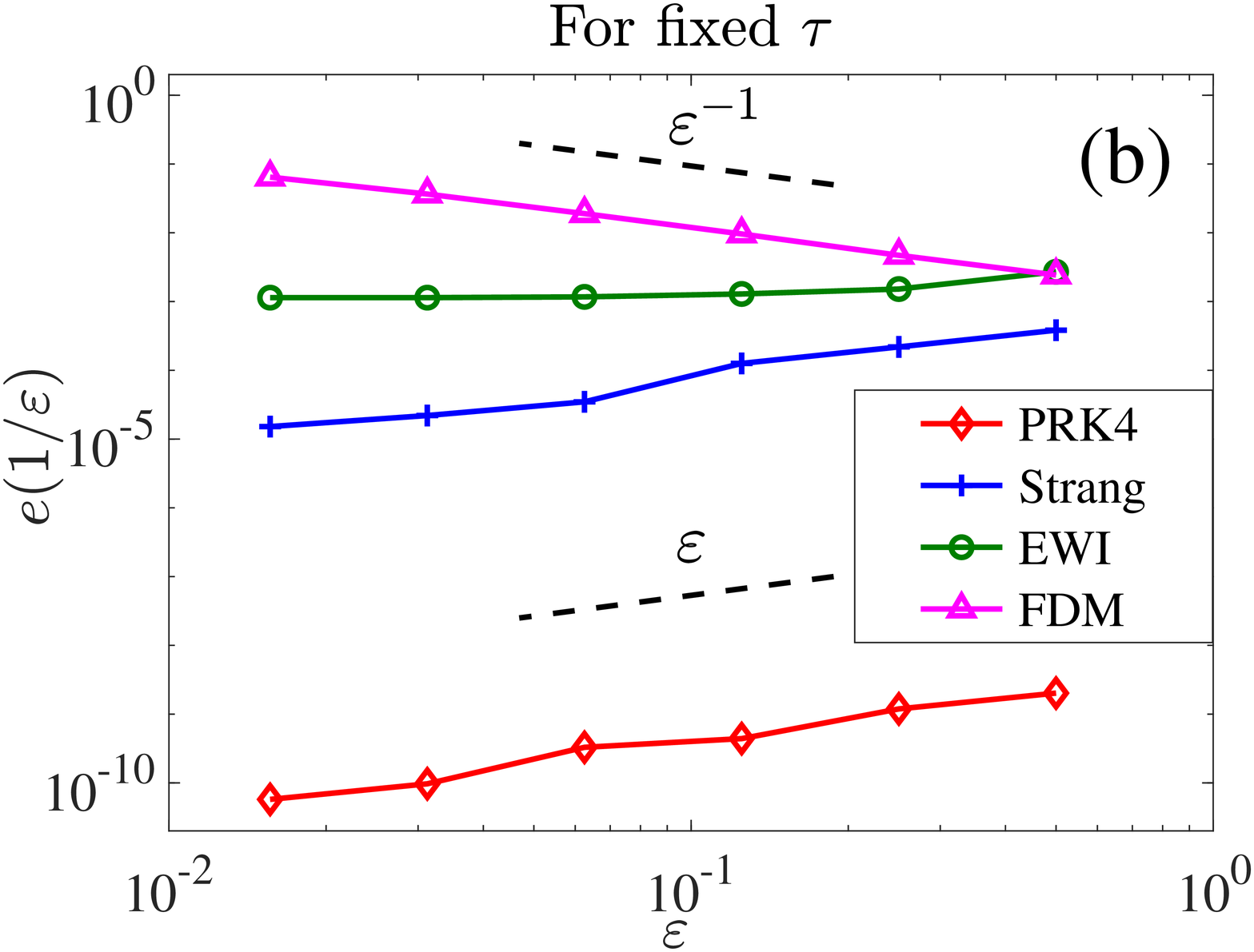}}
\end{minipage}
\caption{Temporal convergence rates of different temporal discretizations  for the Dirac equation \eqref{eq:Dirac_1D} over long-time dynamics at $t = 1/\varepsilon$:
convergence rates in $\tau$ (a), and convergence rates in $\eps$ (b).}
\label{fig:com}
\end{figure}

Fig. \ref{fig:com}(a) displays the temporal errors for the fixed $\eps = 1$ with different time step $\tau$. For the three second-order schemes, the second-order (Strang) time-splitting method  obtains smaller temporal errors than the other two methods with the same time step. The fourth-order time-splitting method (PRK4) not only has higher order convergent rate but also gives much smaller errors than the other three methods with the same time step. Fig. \ref{fig:com}(b) shows the long-time temporal errors of these methods for different $\eps$ with with a fixed time step $\tau$. The splitting methods have improved uniform error bounds like $O(\eps\tau^2)$ up to the long-time at $O(1/\eps)$. The EWI method has uniform error bounds, while the long-time temporal errors of the finite difference method depend on the parameter $\eps$ and behave like $O(1/\eps)$. As a result, time-splitting methods perform much better than  FDM and EWI in the long-time simulations.

\section{Conclusions} \label{sec:conclude}
Improved uniform error bounds for the time-splitting methods for the long-time dynamics of the Dirac equation with small electromagnetic potentials were rigorously established. With the help of the unitary property of the solution flow in $L^2(\Omega)$, the linear growth of the uniform error bound for the time-splitting methods was strictly proven. By employing the regularity compensation oscillation (RCO) technique, the improved uniform error bounds were proved to be $O(\eps\tau^2)$ and $O(h^m + \eps\tau^2)$ up to the long-time at $O(1/\eps)$ for the semi-discretization and full-discretization, respectively. Numerical results were shown to validate our error bounds and to demonstrate that they are sharp. Finally, comparisons of different time discretizations were presented to illustrate the superior property of the time-splitting methods for the numerical simulation of the long-time dynamics of the Dirac equation.

\section*{Acknowledgment}
This work was partially supported by the Ministry of Education of Singapore grant MOE2019-T2-1-063 (R-146-000-296-112, W. Bao \& Y. Feng).

\end{document}